\documentclass[11pt]{amsart}
\usepackage[centering]{geometry}     
\usepackage[en-GB]{datetime2}        
\usepackage{color}
\usepackage[utf8]{inputenc}
\usepackage{fontenc}
\usepackage{enumitem}
\usepackage[backref=page]{hyperref}
\usepackage[nameinlink, capitalize,noabbrev]{cleveref}
\usepackage{graphicx}
\usepackage{amssymb}
\usepackage{amsthm}
\usepackage{amsmath,calligra,mathrsfs}
\usepackage{epstopdf}
\usepackage[all]{xy}
\usepackage[numbers]{natbib}
\definecolor{violet}{rgb}{0.0,0.2,0.7}
\definecolor{rouge2}{rgb}{0.8,0.0,0.2}
\hypersetup{
    bookmarks=true,         % show bookmarks bar?
    unicode=false,          % non-Latin characters in Acrobat's bookmarks
    pdftoolbar=true,        % show Acrobat's toolbar?
    pdfmenubar=true,        % show Acrobat's menu?
    pdffitwindow=false,     % window fit to page when opened
    pdfstartview={FitH},    % fits the width of the page to the window
    pdftitle={},            % title
    pdfauthor={},           % author
    colorlinks=true,        % false: boxed links; true: colored links
    linkcolor=rouge2,       % color of internal links
    citecolor=violet,       % color of links to bibliography
    filecolor=black,        % color of file links
    urlcolor=cyan}          % color of external links

\newtheorem{theorem}{Theorem}[section]
\newtheorem{lemma}[theorem]{Lemma}

\newtheorem{prop}[theorem]{Proposition}

\theoremstyle{definition}
\newtheorem{definition}[theorem]{Definition}

\newtheorem{example}[theorem]{Example}
\theoremstyle{remark}
\newtheorem{rem}[theorem]{Remark}

\title[polarized varieties with high nef value]{A note on polarized varieties with high nef value}
\author{Zhining Liu}
\address{Université Côte d’Azur, CNRS, LJAD, France}
\email{zhining.liu@univ-rennes1.fr, zhining.liu.math@gmail.com}
\date{\today}                                           % Activate to display a given date or no date

\DeclareMathOperator{\Hom}{\mathscr{H}\text{\kern -3pt {\calligra\large om}}\,}
\DeclareMathOperator{\num}{\equiv_{\mathrm{num}}}

\DeclareMathOperator{\NE}{\overline{\mathrm{NE}}}

\DeclareMathOperator{\cond}{\mathfrak{cond}}

\DeclareMathOperator{\reg}{\mathrm{reg}}

\begin{document}

\maketitle
\begin{abstract}
We study the classification problem for polarized varieties with high nefvalue. We give a complete list of isomorphism classes for normal polarized varieties with high nefvalue. This generalizes classical work on the smooth case by Fujita, Beltrametti and Sommese. As a consequence we obtain that  polarized varieties with slc singularities and high nefvalue,  are birationally equivalent to projective bundles over nodal curves.

%{\bf Keywords:} polarized varieties, nefvalue, slc singularities,
\end{abstract}

\tableofcontents

\section{Introduction}
A projective variety $X$ together with an ample line bundle $L$ on $X$ is called a polarized variety and is denoted by $(X,L)$. A classical result on polarized varieties is the Kobayashi-Ochiai theorem:

\begin{theorem}[Generalized Kobayashi-Ochiai Theorem, \protect{\emph{cf.} \citep[Theorem~3.1.6]{BS95}}]
\label{KOG}
Let $X$ be an $n$-dimensional connected normal projective scheme and $L$ an ample line bundle on $X$. Then we have 
\begin{itemize}
\item $(X,L)\cong (\mathbb{P}^n,\mathcal{O}_{\mathbb{P}^n}(1))$ if and only if $K_X+(n+1)L\num \mathcal{O}_X$;
\item $(X,L)\cong (Q,\mathcal{O}_Q(1))$ where $Q\subset \mathbb{P}^{n+1}$ is a hyperquadric in $\mathbb{P}^{n+1}$ if and only if $K_X+nL\num \mathcal{O}_X$.
\end{itemize}
\end{theorem} 

 To study polarized varieties, Fujita introduced the $\Delta$-genus $\Delta(X,L):=n+L^n-h^0(X,L)$ of polarized varieties, which encodes the dimension of the variety $X$ and $L^n$, and develops classification theories for polarized varieties with small $\Delta$-genus under certain assumptions on the singularities of $X$ and positivity on $L$. For Fujita's work, we refer to \citep[Chapter~1]{fujita_1990}.

%The study of polarized varieties has a natural counterpart in the study in the $\mathbb{Q}$-Fano foliations. In \citep{Hor23}, by studying the positivity of twisted cotangent sheaves H\"{o}ring established a Kobayashi-Ochiai type theorem for foliations (\citep[Corollary~1.2]{Hor23}) which gives an upper bound of the index of a $\mathbb{Q}$-Fano foliation $\mathcal{F}$.  

When a foliation $\mathcal{F}$ is algebraically integrable, one can define naturally general log leaves of $\mathcal{F}$ (\emph{cf.} \citep[Definition~3.11]{AD14}). A \textit{general log leaf} $(\tilde{F},\tilde{\Delta})$ comprises a normalization of the closure of a general leaf $F$ of $\mathcal{F}$ and an effective Weil $\mathbb{Q}$-divisor $\tilde{\Delta}$. Let $e:\tilde{F}\rightarrow F$ be the normalization map. Then $\tilde{\Delta}$ is given by $K_{\tilde{F}}+\Delta\num e^{\ast}K_{\mathcal{F}}$. By studying the geometry of general log leaves in \citep{AD14}, Araujo and Druel obtained a version of Kobayashi-Ochiai theorem for $\mathbb{Q}$-Fano foliations (\citep[Theorem~1.2]{AD14}). We also refer to \citep[Corollary~1.2]{Hor23} for a more general statement. %It turns out that understanding the log general leaves  helps to study of algebraically integrable foliations. 
This motivates us to consider classification problem for $(X,\Delta)$, where $X$ is a variety and $\Delta$ a Weil $\mathbb{Q}$-divisor.

When an algebraically integrable foliation $\mathcal{F}$ is $\mathbb{Q}$-Fano, we have the equality \begin{center}
$K_{\tilde{F}}+\Delta \num i_{\mathcal{F}}(e^{\ast}H)$.
\end{center} 
Hence one may very well try to establish a pair version of \cref{KOG}. In fact, Fujino and Miyamoto proved the following:
\begin{prop}[\protect{\citep[Corollary~1.3]{FM21}}]
\label{FM21}
Let $(X,\Delta)$ be a projective semi-log canonical pair such that $X$ is connected.. Assume that $(K_X+\Delta)$ is not nef and that $(K_X+\Delta)\num rD$ for some Cartier divisor $D$ on $X$ with $r>n=\dim(X)$. Then $X$ is isomorphic to $\mathbb{P}^n$ with $\mathcal{O}_X(D)=\mathcal{O}_{\mathbb{P}^n}(-1)$ and $(X,\Delta)$ is Kawamata log terminal.
\end{prop}

The result of Fujino and Miyamoto assumes mild singularites on the  pair $(X,\Delta)$ and a divisibility condition of the log canonical bundle $K_X+\Delta$. However, with a foliation $\mathcal{F}$, its log general leaf $(\tilde{F},\tilde{\Delta})$ is \textit{a priori} just normal. On the other hand, the classical results of classification theory in \citep[Chapter~7.2]{BS95} do not need divisibility assumption. 
However we do need $-K_X$ is very positive. Thus one may try to weaken the conditions and consider the classification problems: 
\begin{enumerate}
\item Classify the triple $(X,\Delta, L)$ where $(X,\Delta)$ is log canonical, $L$ is ample and $K_X+(\dim(X)-1)L \notin \mathrm{Pseff}(X)$;
\item Classify the pair $(X,L)$ where $X$ is a projective variety with singularities wilder than normal, $L$ is ample and $K_X+(\dim(X)-1)L \notin \mathrm{Pseff}(X)$.
\end{enumerate}

In order to achieve these goals, we study the more general class of \textit{quasi-polarized varieties} and follow an approach of Andreatta in \citep{And13}.
%We call a projective variety $X$ together with a nef and big line bundle $L$ a \textit{quasi-polarized variety}.
 For a quasi-polarized variety $(X,L)$ where $X$ is $\mathbb{Q}$-factorial and has canonical singularities, we may run a MMP which contracts all $L$-trivial extremal rays and get a polarized variety $(X',L')$ (see \cref{r}). By using Andreatta's result \cref{b} which describes the general fibers of extremal contractions, we can reduce the problem of classifying $(X',L')$ with high nefvalue to the problem of classifying polarized variety with $\Delta$-genus zero. We have the following classification:

\begin{theorem}
\label{Canonicalcase}
Let $X$ be a variety with canonical $\mathbb{Q}$-factorial singularities and $L$ a nef and big line bundle on $X$. Suppose $K_X+(n-1)L\notin \mathrm{Pseff}(X)$. Then we have one of the following cases:
\begin{enumerate}
\item $(X,L)\sim_{\mathrm{bir}} (\mathbb{P}^n,\mathcal{O}_{\mathbb{P}^n}(1))$;
\item $(X,L)$ is birational equivalent to a $(\mathbb{P}^{n-1},\mathcal{O}_{\mathbb{P}^{n-1}}(1))$-bundle over a smooth curve $C$;
\item $(X,L) \sim_{\mathrm{bir}} (Q,\mathcal{O}_{\mathbb{P}^{n+1}}(1))$, where $Q\subset \mathbb{Q}^{n+1}$ is a hyperquadric;
\item $(X,L)\sim_{\mathrm{bir}} (\mathbb{P}^2,\mathcal{O}_{\mathbb{P}^2}(2))$;
\item $(X,L)\sim_{\mathrm{bir}} C_n(\mathbb{P}^2,\mathcal{O}_{\mathbb{P}^2}(2))$, where $C_n(\mathbb{P}^2,\mathcal{O}_{\mathbb{P}^2}(2))$ is a generalized cone over $(\mathbb{P}^2,\mathcal{O}_{\mathbb{P}^2}(2))$
\end{enumerate}
\end{theorem}

This generalizes the results of Beltrametti and Sommese \citep[Proposition~7.2.2 and Theorem~7.2.4]{BS95}. The drawback of letting $L$ be nef and big is that after running MMP we don't have isomorphism and even have indeterminacies.

For a normal variety $X$, we have modifications $\mu:X'\rightarrow X$ for $X$ such that $X'$ has mild singularities and $K_{X'}$ is $\mu$-ample. A good reference for these modifications is \citep[Chapter~1]{kol13}. For a polarized variety $(X,L)$, with $X$ normal, we may take a canonical modifications $\mu:X'\rightarrow X$ for $X$ and consider the quasi-polarized variety $(X',\mu^{\ast}L)$. By applying the previous result, we have the following classification.

\begin{theorem}
\label{Cl}
Let $(X,L)$ be a polarized normal variety of dimension $n$. Suppose that $K_X$ is $\mathbb{Q}$-Cartier and $K_X+(n-1)L\notin \mathrm{Pseff}(X)$. Then we have one of the following cases:
\begin{enumerate}
\item[(1)] $(X,L)\cong (\mathbb{P}^n,\mathcal{O}_{\mathbb{P}^n}(1))$;

\item[(2.i)] $(X,L)\cong (\mathbb{P}(\mathcal{V}),\mathcal{O}_{\mathbb{P}(\mathcal{V})}(1))$, where $\mathcal{E}$ is a rank $n$ ample vector bundle over a smooth curve $C$;

\item[(2.ii)] $(X,L)\cong C_n(\mathbb{P}^1,\mathcal{O}_{\mathbb{P}^1}(a))$ be a generalized cone with $a \geq 3$;

\item[(3)] $(X,L) \cong (Q,\mathcal{O}_{\mathbb{P}^{n+1}}(1))$, where $Q\subset \mathbb{Q}^{n+1}$ is a hyperquadric;

\item[(4)] $(X,L)\cong (\mathbb{P}^2,\mathcal{O}_{\mathbb{P}^2}(2))$;

\item[(5)] $(X,L)\cong C_n(\mathbb{P}^2,\mathcal{O}_{\mathbb{P}^2}(2))$, a generalized cone over $(\mathbb{P}^2,\mathcal{O}_{\mathbb{P}^2}(2))$.
\end{enumerate}
\end{theorem}
In \cref{Cl}, we note that even if in the proof we have taken a modification, in the resulting list we have isomorphism. The reason is that $L$ is ample and birational equivalences between normal polarized varieties are always isomorphisms.

For a log canonical pair $(X,\Delta)$ with $(K_X+\Delta)+(\dim(X)-1)L\notin \mathrm{Pseff}(X)$, a first observation is that if $\Delta$ is $\mathbb{Q}$-Cartier, we will have $K_X+(\dim(X)-1)L\notin \mathrm{Pseff}(X)$. Hence we will have a list for $(X,L)$ similar to \cref{Cl}. However in this list the Picard number $\rho(X)\leq 2$. Hence for $\Delta$ to be an irreducible divisor or more generally reduced divisor, we don't have to many choice. We may thus give a list for $(X,\Delta, L)$.

\begin{prop}
\label{Clr}
Let $(X,\Delta)$ be a log canonical pair, with $\Delta\neq 0$ a reduced divisor. Suppose that $L$ is an ample line bundle on $X$ and $(K_X+\Delta)+(n-1)L\notin \mathrm{Pseff}(X)$, where $n=dim(X)$. Then $(X,\Delta,L)$ is one of the following:
\begin{enumerate}
\item $(X,L)\cong (\mathbb{P}^n,\mathcal{O}_{\mathbb{P}^n}(1))$, $\Delta\equiv_{\mathrm{num}} H$ is a prime divisor where $H$ is a hyperplane of $ \mathbb{P}^n$;

\item[(2.i)] There is a $(\mathbb{P}^{n-1},\mathcal{O}_{\mathbb{P}^{n-1}}(1))$-bundle $(\mathbb{P}(E),\mathcal{O}_{\mathbb{P}(E)}(1))$ over a smooth curve $C$, and a birational morphism $\mu:\mathbb{P}(E)\rightarrow X$ such that $\mu^{\ast}(L)\cong \mathcal{O}_{\mathbb{P}(E)}(1)$ and $\Delta=\sum F_i$ is a finite sum where $F_i\cong \mu(\mathbb{P}^{n-1})$ are images of distinct general fibers of $\pi$ by $\mu$;

\item[(2.ii)] $(X,L)=(\mathbb{P}(\mathcal{O}_{\mathbb{P}^1}(a)\oplus\mathcal{O}_{\mathbb{P}^1}(1)),\mathcal{O}_{\mathbb{P}(\mathcal{O}_{\mathbb{P}^1}(a)\oplus\mathcal{O}_{\mathbb{P}^1}(1))}(1))$ with $a>1$ and $\Delta=D$ is irreducible, where $D$ is the unique section of $\mathbb{P}(\mathcal{O}_{\mathbb{P}^1}(a)\oplus\mathcal{O}_{\mathbb{P}^1}(1))\rightarrow \mathbb{P}^1$ such that $D\equiv_{\mathrm{num}}\mathcal{O}_{\mathbb{P}(\mathcal{O}_{\mathbb{P}^1}(a)\oplus\mathcal{O}_{\mathbb{P}^1}(1))}(1))-af$, where $f$ is a general fiber;

\item[(3.i)] $(X,L)\cong (Q,\mathcal{O}_{\mathbb{P}^{n+1}}(1))$, where $Q\subset \mathbb{P}^{n+1}$ is a $\mathrm{rk}(Q)=3$ hyperquadric, the boundary divisor $\Delta$ is a hyperplane in $Q$ and  $[\Delta]=\dfrac{1}{2}[H\cap Q]$ where $H$ is a hyperplane in $\mathbb{P}^{n+1}$;

\item[(3.ii)] $(X,L)\cong (Q,\mathcal{O}_{\mathbb{P}^{n+1}}(1))$, where $Q\subset \mathbb{P}^{n+1}$ is a $\mathrm{rk}(Q)=4$ hyperquadirc. If we write $Q=\mathrm{Proj}\left( \dfrac{\mathbb{C}[x_0, \ldots, x_{n+1}]}{(x_0 x_1-x_2 x_3)}\right) $, then $\Delta=D$ is prime and $D$ is the cone with vertex $\mathbb{P}^{n-3}$ over $\mathbb{P}^{1}\times \mathrm{pt}$ or $\mathrm{pt}\times \mathbb{P}^1$. In particular, $D\cong \mathbb{P}^{n-1}$.
\end{enumerate}
\end{prop}

Finally we turn to non-normal varieties with semi-log canonical singularities. We have the following classification.

\begin{theorem}
\label{Slccase}
Let $X$ be a non-normal slc projective variety of dimension $n$ and $L$ an ample line bundle over $X$. Suppose that $K_X+(n-1)L\notin\mathrm{Pseff}(X)$. Let $\pi:\bar{X}\rightarrow X$ be the normalization of $X$ and $D\subset X$, $\bar{D}\subset \bar{X}$ the conductors. Then we have:

There is a nodal curve $C'$,a rank $n$-vector bundle $E'$, distinct fibers $F_1, F_2, \ldots, F_m$ of $\mathbb{P}(E')$ and a birational morphism $\mu:\mathbb{P}(E')\rightarrow X$ such that $\mu^{\ast}(L)=\mathcal{O}_{\mathbb{P}(E')}(1)$ and $D=\sum_{1\leq i\leq m}\mu(F_i)$

\end{theorem}

We see that \cref{Slccase} shortens the list in \cref{Clr} rather than increasing it. In fact there is a degree $2$ morphism $\bar{D}^\nu\rightarrow D^\nu$, where $\bar{D}^\nu$ and $D^\nu$ are the normalizations of $\bar{D}$ and $D$ respectively. Hence we need $(L'|_{\bar{D}^\nu})^{n-1}$ to be divisible by $2$ which gives more restrictions on $(\bar{X},\bar{D})$ than the assumption in \cref{Clr}.

\begin{rem}
%After running MMP to reduce \cref{Canonicalcase} to the problem of classifying $(X',L')$ with $X'$ canonical and $L'$ ample (see \cref{r}), 
The classification in \cref{Canonicalcase} is already known for even when $X'$ is klt (\emph{cf.} \citep[Proposition~3.5]{And13}). My personal contribution in the classification is to use modifications to get \cref{Cl} and \cref{Slccase}.
\end{rem}

\subsection{Plan of the article}
The article is organized as following. 
In \cref{convA}, we recall some basic notions and facts that we need. In  \cref{cpv}, we prove \cref{Cl} by running an MMP (\cref{r}) to reduce the problem to check which member in the list of classification results of Fujita, Beltrametti-Sommese satisfies our non pseudo-effective hypothesis. In  \cref{npv}, we prove \cref{Cl} thanks to canonical modifications and use similar methods to prove \cref{Clr}. In  \cref{slcpv}, for a polarized slc variety $(X,L)$, we use \cref{Clr} on the triple $(\bar{X},\bar{D},L')$, where $(\bar{X},\bar{D})$ is the normalization of $X$ and the conductor divisor on $\bar{X}$ and $L'$ is the pullback of $L$, to get \cref{Slccase}.

\subsection*{Acknowledgement}
The article is part of the my PhD thesis. I would like to express my sincere gratitude to my advisor Andreas Höring for his support and guidance, and for proposing this project to me. I also thank Enrica Floris for reading my rather awkward draft and making many suggestions to make the article more readable. I also thank my co-advisor Benoît Claudon for pointing out some errors and typos. The essential part of this article was written when I was in Laboratoire J.A. Dieudonné. I thank the LJAD for its accommodation and its nice research environment.

\section{Notations and general setup}
\label{convA}
We work over $\mathbb{C}$. For general definitions we refer to \citep{har78}.

 A scheme in the article will always be projective over $\mathbb{C}$. A variety is a reduced and irreducible scheme over $\mathbb{C}$. The name \textit{point} does not necessarily refer to closed point.
 
 For two Cartier divisor $D_1$ and $D_2$, we denote by $D_1\sim D_2$ the linear equivalence and by $D_1\num D_2$ the numerical equivalence. We have similar notations for $\mathbb{Q}$-Cartier Weil-divisors.

A vector bundle $\mathcal{V}$ of rank $r$ over $X$ is a locally free sheaf of rank $r$. We set \begin{equation*}
\mathbb{P}(\mathcal{V}):=\mathrm{Proj}(\oplus_{n\geq 0}\mathrm{Sym}^n(\mathcal{V}))
\end{equation*}
to be its \textit{projectivisation}.

  We follow the positivity notions of divisors and vector bundles in \citep{Laz03}\citep{Laz04}. We will use a generalized version of \textit{pseudo-effectivenss} for reflexive sheaves by Höring-Peternell:
  \begin{definition}[\protect{\citep[Definition~2.1.]{HP18}}]
\label{pseff}
Let $X$ be a normal projective variety and $\mathcal{E}$ a reflexive sheaf on $X$. We say that $\mathcal{E}$ is \textit{pseudo-effective} if there exists an ample divisor $H$ on $X$ satisfying the following:
For any $c>0$ there exists integers $j>0$ and $i>jc$ such that 
\begin{equation*}
H^0(X,S^{[i]}(\mathcal{E})\otimes \mathcal{O}_X(jH))\neq 0
\end{equation*}
where $S^{[i]}(\mathcal{E})$ is the double dual of $\mathrm{Sym}^i(\mathcal{E})$.
\end{definition}
When $\mathcal{E}$ itself is a line bundle, we have that  $\mathcal{E}$ is pseudo-effective in the above sense is equivalent to $\mathcal{E}$ is pseudo-effective in the usual sense of \citep[Definition~2.2.25]{Laz03}.

We recall the definition of polarized and quasi-polarized varieties.
\begin{definition}
Let $(X,L)$ be a pair consisting of a projective variety $X$ and a line bundle $L$ over $X$. We call it 
\begin{enumerate}
\item a \textit{quasi-polarized variety} if $L$ is nef and big;
\item a \textit{polarized variety} if $L$ is ample.
\end{enumerate}
For a quasi-polarised variety $(X,L)$ of dimension $n$, its \textit{$\mathrm{\Delta}$-genus} is defined to be
\begin{equation*}
\label{deltagenus}
\mathrm{\Delta}(X,L):=n+L^n-h^0(X,L)
\end{equation*}
\end{definition}

Let $(X,L)$ be a quasi-polarized variety, we define the \textit{nefvalue} $\tau(L)$ of $L$ to be
\begin{equation*}
\tau(L):=\inf\{t \in \mathbb{R} :K_X+tL \text{ is nef}\}.
\end{equation*}
By Kawamata's rationality theorem, we know that $\tau(L)$ is a rational number or $\infty$.

We now give our notions for birational equivalence and isomorphisms between quasi-polarized varieties.
\begin{definition}
\label{bir}
Let $(X_1,L_1)$ and $(X_2,L_2)$ be two pairs consisting of a variety $X_i$ and a line bundle $L_i$ on $X_i$. We say that
\begin{enumerate}
\item $(X_1,L_1)$ is \textit{isomorphic} to $(X_2,L_2)$, if there exists an isomorphism $\phi:X_1\rightarrow X_2$ such that $\phi^{\ast}(L_2)$ is isomorphic to $L_1$. We denote this by $(X_1,L_1)\cong (X_2,L_2)$.
\item $(X_1,L_1)$ and $(X_2,L_2)$ are \textit{birationally equivalent}, if there exists a variety $X$ and two birational morphism $\phi_i:X\rightarrow X_i$ such that $\phi_1^{\ast}(L_1)$ is isomorphic to $\phi_2^{\ast}(L_2)$. We denote this by $(X_1,L_1)\sim_{\mathrm{bir}} (X_2,L_2)$.
\end{enumerate}
\end{definition}

In the article we will repeatedly encounter \textit{generalized cones}. We thus recall the notion of generalized cone here.
\begin{definition}[Generalized cone]
\label{gene-cone}
We follow the construction in \citep[1.1.8.]{BS95}
Let $V$ be a projective scheme of dimension $n$ and $L$ a very ample line bundle over $V$. Fix $N\geq n$ an integer. Set $E:=\oplus^{N-n}\mathcal{O}_V$ and $p:\mathbb{P}(E\oplus L)\rightarrow V$. We denote $\mathbb{P}(E\oplus L)$ by $X$.  Note that $E\oplus L$ is globally generated and we have for the tautological bundle $\xi:=\mathcal{O}_{\mathbb{P}(E\oplus L)}(1)$ of $\mathbb{P}(E\oplus L)$ a surjective morphism $p^{\ast}(E\oplus L)\rightarrow \xi$. Hence we have a surjective morphism
\begin{center}
$H^0(V,E\oplus L)\otimes_{\mathbb{C}}\mathcal{O}_X\twoheadrightarrow\xi$.
\end{center}
The above morphism corresponds to a unique morphism 
\begin{center}
$\phi_{\vert\xi\vert}:X\rightarrow \mathbb{P}(H^0(V,E\oplus L))$.
\end{center}
We take the Stein factorization of $\phi$:
\begin{center}
$\xymatrix{X\ar[r] \ar[d]_{\psi_{|\xi|}} & \mathbb{P}(H^0(V,E\oplus L))\\
 C_N(V,L)\ar[ru]}$
\end{center}
and call $C_N(V,L)$ the \textit{generalized cone of dimension N} on $(V,L)$. As $\xi$ is big, the scheme $C_N(V,L)$ has dimension $N$. Set $\xi_L:=\mathcal{O}_{\mathbb{P}(H^0(V,E\oplus L))}(1)|_{C_N(V,L)}$, then $\xi_L$ is ample.
\end{definition}

For the notions and results in birational geometry and the minimal model program, we refer to the standard \citep{KM98} and \citep{kol13}.

\section{Canonical polarized varieties}
\label{cpv}
In this section, we consider quasi-polarized varieties $(X,L)$ with canonical singularities. First we give a lemma to show how the condition $K_X+(n-1)L\notin \mathrm{Pseff}(X)$ is related to the nefvalue of $L$.

\begin{lemma}
\label{a}
Let $(X,L)$ be quasi-polarized variety of dimension $n$ with canonical singularities. Suppose that $\tau(L)$ is finite. If $K_X+(n-1)L\notin \mathrm{Pseff}(X)$, we have that $\tau(L)>n-1$.
\end{lemma}
\begin{proof}
We know that $\mathrm{Pseff}(X')=\overline{\mathrm{Big}(X)}$ is a closed cone. Hence there exists an ample $\mathbb{Q}$-divisor $A$, such that $K_{X}+(n-1)L+A$ is not pseudo-effective. If $\tau(L)\leq n-1$, we have 
\begin{equation*}
K_{X}+(n-1)L+A=(K_{X}+\tau(L)L)+(n-1-\tau(L))L+A.
\end{equation*}
That is, $K_X+(n-1)L+A$ is a sum of a nef and an ample divisor, which is ample, a contradiction.
\end{proof}

When $L$ is ample, its nefvalue $\tau(L)$ is finite. However, when $L$ is just nef and big, we have some subtleties. By the cone theorem (\emph{cf.} \citep[Theorem~1.1.]{F11}), we know that
\begin{equation*}
\NE(X)=\NE(X)_{K_X\geq 0}+\sum \mathbb{R}_{\geq 0}[C_j]
\end{equation*}
where $C_j$ are $K_X$-negative rational curves and the sum is over countably many $j$.

\begin{enumerate}
\item For every $K_X$-negative extremal ray $R=\mathbb{R}_{\geq 0}[C]$, we have that $L \cdot C >0$. By the rationality theorem (\emph{cf}. \citep[Complement~3.6]{KM98}), there exists a $K_X$-negative extremal $C_0$ such that $\tau(L)=-\dfrac{K_X\cdot C_0}{L\cdot C_0}<\infty$. Hence $\tau(L)=\infty$ only if there exists an $L$-trivial $K_X$-negative extremal ray.

\item There exists a $K_X$-negative extremal ray $R$ such that $L\cdot R=0$. By the contraction theorem (\emph{cf.} \citep[Theorem~1.1.(4)]{F11}), we consider the contraction with respect to $R$, $\mathrm{cont}_R:X\rightarrow Z$. Note that there exists a line bundle $L_Z$ on $Z$ such that $L\cong \mathrm{cont}_R^{\ast}(L_Z)$.
\end{enumerate}

Hence we may consider to run a $K_X$-MMP to contract every $L$-trivial extremal rays to get a $(X',L')$ satisfying the case $(1)$. We now precise how to do this.

\begin{lemma}
\label{r}
Let $X$ be a variety with canonical $\mathbb{Q}$-factorial singularities and $L$ a big and nef line bundle on $X$. Suppose that $K_X+(n-1)L\notin \mathrm{Pseff}(X)$. Then $(X,L)$ is birationally equivalent to a quasi-polarized variety $(X',L')$, where $X'$ is a normal projective variety with canonical $\mathbb{Q}$-factorial singularities, $K_{X'}+(n-1)L'\notin \mathrm{Pseff}(X')$ and 
\begin{enumerate}
\item Either $\tau(L')$ is finite;
\item or there is a Mori fiber space structure $\phi:X'\rightarrow W$ and a rational number $\tau>(n-1)$ such that $L'$ is $\phi$-ample and $K_{X'}+\tau L'\sim_{\mathbb{Q},\phi}0$.
\end{enumerate}
\end{lemma}
\begin{proof}
We apply the terminal modification then a small $\mathbb{Q}$-factorialization to $X$ (\emph{cf.}\citep[Theorem~1.33, Corollary~1.37]{kol13}). We get a modification $f:Y\rightarrow X$ such that $Y$ has $\mathbb{Q}$-factorial terminal singularities. Set $L_Y=f^{\ast}L$. We have that $L_Y$ is nef and big and $K_Y+(n-1)L_Y\notin \mathrm{Pseff}(Y)$. By \citep[Lemma~4.1.]{And13}, we can find an effective $\mathbb{Q}$-divisor $\Delta$ on $Y$ such that 
\begin{center}
$\Delta\sim_{\mathbb{Q}} (n-1)L_Y$ and $(Y,\Delta)$ is klt.
\end{center}
Now consider the pair $(Y,\Delta)$. We have that $K_Y+\Delta \notin \mathrm{Pseff}(Y)$. By \citep[Corollary~1.3.3]{BCHM}, we can run a $(K_Y+\Delta)$-MMP to get 
\begin{equation*}
\label{MMP}
(Y,\Delta)=(Y_0,\Delta_0)\dashrightarrow (Y_1,\Delta_1)\dashrightarrow\cdots\dashrightarrow (Y_s,\Delta_s),
\end{equation*}
with $Y_s$ a Mori fiber space.

Suppose that the map $\phi_i:Y_i\dashrightarrow Y_{i+1}$ is associated with a $(K_{Y_i}+\Delta_i)$-negative extremal ray $R_i$. By \citep[Proposition~4.2.]{And13}, for every $i=0, 1,\ldots, s$, we have that
\begin{enumerate}
\item $Y_i$ is $\mathbb{Q}$-factorial terminal;
\item $\Delta_i\cdot R_i=0$;
\item There exists nef and big line bundles $L_i$ on $Y_i$ and $\Delta_i\sim_{\mathbb{Q}} (n-1)L_i$.
\end{enumerate}
It is then obvious $K_{Y_i}+(n-1)L_i\notin \mathrm{Pseff}(Y_i)$.

We then set $(X',L'):=(Y_s, L_s)$. 
\begin{enumerate}
\item If $(Y_s,\Delta_s)$ has no $K_{Y_s}$-negative extremal ray $R$ such that $L_s\cdot R=0$, by Kawamata rationality theorem there exists a $K_{X'}$-negative extremal curve $C_0$  such that $\tau(L')=-\dfrac{K_{X'}\cdot C_0}{L'\cdot C_0}$. Hence the nefvalue of $L'$ is finite.

\item Otherwise, we consider the Mori fiber space $\phi_s:Y_s\rightarrow W$ obtained in the above $(K_Y+\Delta)$-MMP. Let $R_s:=\mathrm{NE}(\phi_s)$ be the extremal ray of $\phi_s$. We claim that $L_s\cdot R_s > 0$. Suppose by contradiction that $L_s\cdot R_s=0$. Then by the contraction theorem,  there exists $L_W$ such that $\phi_s^{\ast}(L_W)=L_s$.  As $\dim(W)<\dim(Y_s)$, we have that $L_s^n=\phi_s^{\ast}(L_W^n)=0$ contradicting $L_s$ to be nef and big. As $R_s$ is a $(K_{Y_s}+\Delta_s)$-negative extremal ray, we have that $(K_{Y_s}+(n-1)L_s)\cdot R_s<0$. Hence the $\tau>0$ such that $K_{Y_s}+\tau L_s\sim_{\mathbb{Q},\phi} 0$ satisfies that $\tau>(n-1)$.
\end{enumerate}
\end{proof}

The following relative Kobayashi-Ochiai criterion by Andreatta is the major tool to give further classification for the second situation in the above lemma.

\begin{theorem}[\protect{\cite[Theorem~2.1.]{And95}}]
\label{b}
Let $X$ be a projective variety with klt singularities and let $L$ be a line bundle on $X$. Let $\phi:X\rightarrow Z$ be a surjective morphism with connected fibers between normal varieties. Suppose that $L$ is $\phi$-ample and $K_X+\tau L\sim_{\mathbb{Q},\phi} 0$ for some $\tau\in \mathbb{Q}^{+}$. Let $F_1=\phi^{-1}(z)$ be a non-trivial fiber, $F\subset F_1$ be one of its irreducible components, $F'$ be the normalization of $F$ and let $L'$ be the pullback of $L$ on $F'$. Let $\lfloor \tau\rfloor$ be the integral part of $\tau$ and $ \tau'=\lceil \tau \rceil=-\lfloor -\tau\rfloor$.
\begin{enumerate}
\item[(I,1)]  $\mathrm{dim}(F)\geq \tau-1$;
\item[(I,2)] If $\mathrm{dim}(F)< \tau$, then $F\cong \mathbb{P}^{\tau'-1}$ and $L|_F=\mathcal{O}_{\mathbb{P}^{\tau'-1}}(1)$;
\item[(I,3)] If $\mathrm{dim}(F)<\tau+1$, then $\Delta(F',L')=0$,
\end{enumerate}
If moreover $\mathrm{dim}(F)>\mathrm{dim}(X)-\mathrm{dim}(Z)$, then
\begin{enumerate}
\item[(II,1)]  $\mathrm{dim}(F)\geq \tau$;
\item[(II,2)] If $\mathrm{dim}(F)= \tau$, then $F\cong \mathbb{P}^{\tau}$ and $L|_F=\mathcal{O}_{\mathbb{P}^{\tau}}(1)$;
\item[(II,3)] If $\mathrm{dim}(F)<\tau+1$, then $\Delta(F',L')=0$,
\end{enumerate}
 If all components of the fiber $F_1$ satisfy that $\mathrm{dim}(F)<\tau$, in case (I.2) or $\mathrm{dim}(F)\leq \tau$ in case (II.3), then the fiber is actually irreducible.
\end{theorem}

A direct result of the above Theorem is the following lemma which classifies the $(X',L')$ in the case $(2)$ of \cref{r}.

\begin{lemma}
\label{cl1}
Let $(X,L)$ be a quasi-polarized variety of dimension $n$. Suppose $X$ has canonical $\mathbb{Q}$-factorial singularities and $K_X+(n-1)L\notin \mathrm{Pseff}(X)$. Suppose that there exists a $K_X$-negative extremal ray $R=\mathbb{R}_{\geq0}[C_0]$ such that $L\cdot C_0>0$.
Then $(X,L)$ is the one of the following
\begin{enumerate}
\item $(X,L)\cong (\mathbb{P}^n,\mathcal{O}_{\mathbb{P}^n}(1))$, and $\tau=n+1$;
\item $(X,L)$ is isomorphic to a $(\mathbb{P}^{n-1},\mathcal{O}_{\mathbb{P}^{n-1}}(1))$-bundle over a smooth curve $C$ and $\tau=n$;
\item $\Delta(X,L)=0$, $K_X+\tau L\sim_{\mathbb{Q}} \mathcal{O}_X$ and $n-1<\tau\leq n$.
\end{enumerate}
\end{lemma}

\begin{proof}
Let $\phi:X\rightarrow Z$ be the Mori contraction of the extremal ray $R$. Set $t>0$ to be the rational number such that $(K_X+tL)\cdot C_0=0$. Let $F$ be a general fiber of $\phi$, then $(K_X+(n-1)L)|_F\notin \mathrm{Pseff}(F)$. As $\NE(F)=\mathbb{R}_{\geq0}[C_0]$, we have that $(K_X+(n-1)L)\cdot C_0<0$. Thus $t>(n-1)$.

Let $m$ be a divisible enough integer such that $m K_X$ is a Cartier divisor and $mt$ is an integer. The line bundle $mK_X+mt L$ is $\phi$-numerically trivial. By the contraction theorem, we know that $K_X+t L\sim_{\mathbb{Q},\phi}0$. As $\mathrm{NE}(X/Z)=R$, we have that $L$ is $\phi$-ample. Thus we are in the situation of \cref{b}.
 
 We first show that $\phi$ is not birational. Suppose by contradiction that $\phi:X\rightarrow Z$ is birational. Let $F$ be a component of a non trivial fiber $F_1=\phi^{-1}(z)$.  By \cref{b} (II,1), we have that $\mathrm{dim}(F)\geq t>n-1$. Thus $\phi(X)$ is a singleton, a contradiction.
 
  By \cref{b}, we know that $\mathrm{dim}(F)\geq t-1>n-2$. Thus we have that either $\mathrm{dim}(F)=n$ or $\mathrm{dim}(F)=n-1$.
 \begin{enumerate}
 \item If $\mathrm{dim}(F)=n$, we have that $F=X$ and $Z=\{z\}$. Then $K_X+tL\sim_{\mathbb{Q}} \mathcal{O}_X$ and $\tau=t$. If $t>n$, \cref{b} (I.2) implies that $(X,L)=(\mathbb{P}^n,\mathcal{O}_{\mathbb{P}^n}(1))$ and $\tau=n+1$. If $n-1<t\leq n$, we have that $\dim(F)=n<t+1$. By \cref{b} (I.3), we know that $\Delta(X,L)=0$.
 
 \item  $\mathrm{dim}(F)=n-1$. Let $F'\subset F_1$ be another component of $F_1$. Then \cref{b} implies $\dim(F')\geq n-1$. On the other hand we can not have $\dim(F')=n$, for this would imply that $F=F'=X$ which has dimension $n$, a contradiction. Hence by \cref{b} again, we know that $F_1$ is irreducible and $F=F_1$.  As $\phi$ is not birational, by semi-continuity of dimensions of fibers (\emph{cf.} for example  \cite[\href{https://stacks.math.columbia.edu/tag/02FZ}{Tag 02FZ}]{stacks-project}), for any point $z'$, the fiber $\phi^{-1}(z')$ has positive dimension. By \cref{b} and repeating the argument for $F$ and $F_1$, we know that $\phi^{-1}(z')$ is irreducible with dimension $n-1$. Then \cref{b} (I,2) implies that for every fiber $\phi^{-1}(z')$, we have that $(\phi^{-1}(z'),L_{\phi^{-1}(z')})\cong (\mathbb{P}^{n-1},\mathcal{O}_{\mathbb{P}^{n-1}}(1))$. Thus we know that $(X,L)$ is isomorphic to a $(\mathbb{P}^{n-1},\mathcal{O}_{\mathbb{P}^{n-1}}(1))$-bundle over a smooth curve $C$ and $\tau=n$.
 \end{enumerate} 
\end{proof}

We are now left in the case ($3$) of \cref{cl1}. In this case, we have the following:

\begin{lemma}
\label{cl2}
Let $(X,L)$ be a quasi-polarized variety of dimension $n$ with $\Delta(X,L)=0$. Suppose that $X$ has canonical $\mathbb{Q}$-factorial singularities, and that the nefvalue $\tau=\tau(L)$ of $L$ satisfies $n-1<\tau(L)\leq n$. If $K_X+\tau L\sim_{\mathbb{Q}} \mathcal{O}_X$, then there exists a birational morphism $\mu:X\rightarrow Y$ such that 
\begin{enumerate}
\item $Y$ has canonical singularities, $\mu^{\ast}(K_Y)=K_X$ ;
\item there exists an ample line bundle $A$ on $Y$ such that $\mu^{\ast}(A)=L$;
\item $\Delta(Y,A)=0$ and $K_Y+\tau A\equiv_{\mathrm{num}} \mathcal{O}_Y$.
\end{enumerate}
\end{lemma}
\begin{proof}
We have that $L-K_X\sim_{\mathbb{Q}} 2\tau L$ which is nef and big. Hence we may apply the basepoint-free theorem for $L$ (\cite[Theorem~ 3.3.]{KM98}), to get that for all sufficient large integer $b$, the linear system $|bL|$ has no basepoints. We fix a such integer $b_0$. Now consider the graded algebra 
\begin{equation*}
R(X,L)=:\bigoplus_{n\geq 0} H^0(X,nL).
\end{equation*}
We have a canonical rational map $\mu:X\rightarrow \mathrm{Proj}(R(X,L))=:Y$. As $\mathrm{Bs}(|b_0L|)=\emptyset$, we know that $\mu$ has no indeterminacy and $R(X,L)$ is finite generated (\emph{cf.} \citep[Proposition~7.6.]{deb}). Hence the ring $R(X,b_0L)$ is integral and normal. As $L$ is big, the morphism $\mu$ is birational and $b_0L=\mu^{\ast}(A_1)$(\emph{cf.} \citep[Lemma~7.10.]{deb}) for some ample Cartier divisor $A_1$. With the same argument for the integer $b_0+1$, we get another ample Cartier divisor $A_2$ such that $(b_0+1)L=\mu^{\ast}(A_2)$.  By setting $A:=A_2-A_1$, we get (2).

We now take a divisible enough $m$ such that $mK_X$ is Cartier, the number $m\tau$ is an integer and $mK_X+m\tau L\sim_{\mathbb{Z}} 0$. Denote by $E$ the exceptional locus of $\mu$  and by $\nu:Y\setminus \mu(E)\rightarrow X\setminus E$ the inverse of $\mu$. We have that
\begin{equation*}
\mathcal{O}_Y(mK_Y)|_{Y\setminus \mu(E)}\sim \nu^{\ast}(\mathcal{O}_X|_{X \setminus E})\sim \nu^{\ast}(-m\tau L|_{X\setminus E})\sim -m\tau A|_{Y\setminus \mu(E)}
\end{equation*}
 We have that the rank one reflexive sheaf $\mathcal{O}_Y(mK_Y)$ and the line bundle $-m\tau A$ agree outside a subset whose codimension is at least $2$. Hence $\mathcal{O}_Y(mK_Y)$ is a line bundle and $K_Y$ is $\mathbb{Q}$-Cartier. We thus have the equalities $K_Y=-\tau A$ and $\mu^{\ast}(K_Y)=K_X$. Hence $\mu$ is crepant and $Y$ has canonical singularities. We get ($1$). By projection formula, we have that
 \begin{center}
 $K_Y+\tau A=\mu_{\ast}(K_X+\tau L)=\mathcal{O}_Y$.
 \end{center}
 Thus $\Delta(Y,A)=n+A^n-h^0(Y,A)=n+L^n-h^0(X,L)=0$.
\end{proof}

Hence it rest for us to classify the polarized variety $(X,L)$ with $L$ ample, $n-1<\tau(L)\leq n$, $\Delta(X,L)=0$ and $K_X+\tau(L)L\cong \mathcal{O}_X$. We have the following

\begin{lemma}
\label{cl3}
Let $(X,L)$ be a polarized variety with $L$ ample, $n-1<\tau(L)\leq n$, $\Delta(X,L)=0$ and $K_X+\tau(L)L\equiv_{\mathrm{num}} \mathcal{O}_X$. Suppose that $X$ has canonical singularities. Then one of the following occurs:
\begin{enumerate}
\item $(X,L)\cong (Q,\mathcal{O}_{\mathbb{P}^{n+1}}(1))$, where $Q\subset \mathbb{Q}^{n+1}$ is a hyperquadric;
\item $(X,L)$ is a $\mathbb{P}^{n-1}$-bundle over $\mathbb{P}^1$ and the restriction of $L$ to each fiber is $\mathcal{O}_{P^{n-1}}(1)$;
\item $(X,L)\cong (\mathbb{P}^2,\mathcal{O}_{\mathbb{P}}^2(2))$;
\item $(X,L)\cong C_n(\mathbb{P}^2,\mathcal{O}_{\mathbb{P}^2}(2))$ is a generalized cone over $(\mathbb{P}^2,\mathcal{O}_{\mathbb{P}^2}(2))$
\end{enumerate}
\end{lemma}

\begin{proof}
If $\tau(L)=n$, we have that $K_X+nL\equiv_{\mathrm{num}} \mathcal{O}_X$. Then \cref{KOG} implies that $(X,L)\cong (Q,\mathcal{O}_{\mathbb{P}^{n+1}}(1))$, where $Q\subset \mathbb{P}^{n+1}$ is a hyperquadric. Hence we are in case ($1$). 

From now on we may assume that $\tau(L)<n$. As $L$ is ample, we have that
\begin{equation*}
K_X+nL\num(n-\tau(L))L
\end{equation*}
is ample.

By Fujita's classification theorem for polarized varieties with $\Delta$-genus zero (\emph{cf.}
\citep[Theorem~5.10 and Theorem~5.15]{fujita_1990} \citep[Proposition~3.1.2.]{BS95}), we know that besides the four cases given above in \cref{cl3}, there are two more possibilities for $(X,L)$:
\begin{enumerate}
\item[(i)] Either $(X,L)\cong (\mathbb{P}^n,\mathcal{O}_{\mathbb{P}^n}(1))$,
\item[(ii)] or $(X,L)$ is a generalized cone over $(V,L_V)$, where $V\subset X$ is a smooth submanifold, $L|_V=L_V$ is very ample and $\Delta(V,L_V)=0$.
\end{enumerate}
Case (i) is impossible, since $\tau(\mathcal{O}_{\mathbb{P}^n}(1))=n+1$. Hence we need to investigate case (ii). Set $r:=n-\mathrm{dim}(V)$. From \cref{gene-cone} we have the following diagram
\begin{center}
$\xymatrix{
&\mathbb{P}(\mathcal{O}_V^{\oplus r})=V\times\mathbb{P}^{r-1} \ar[r]^-{\mathrm{pr}_2}\ar@{^{(}->}[d]^i & \mathbb{P}^{r-1}  \ar@{^{(}->}[d]\\
\mathbb{P}(L_V)\ar[r]\ar[rd]_{\cong}& \mathbb{P}(\mathcal{O}_V^{\oplus r}\oplus L_V)\ar[d]^{\pi}\ar[r]^{\psi_{|\xi|}} &C_n(V,L_V)=X \\ 
& V
}$
\end{center}
where $\xi=\mathcal{O}_{\mathbb{P}(\mathcal{O}_V^{\oplus r}\oplus L_V)}(1)$ is the tautological bundle. The identification of $V\cong \mathbb{P}(L_V)$ is given by the quotient morphism $\mathcal{O}_V^{\oplus r}\oplus L_V\twoheadrightarrow L_V$.

We claim that outside $\mathbb{P}(\mathcal{O}_V^{\oplus r})$ the morphism $\psi_{|\xi|}$ induces an isomorphism onto its image. Take $z\in C_n(V,L)$ such that $\psi_{|\xi|}^{-1}(z)$ has positive dimension. In particular, there exists a curve $C_1$ such that $\psi_{|\xi|}(C_1)=\{z\}$. Since $\mathcal{O}_V^{\oplus r}\oplus L_V$ is globally generated, we know that $\psi_{|\xi|}$ restricted to each fiber of $\pi$ is an embedding. Hence $\pi$ maps $C_1$ bijectively to its image $C$. By generic smoothness, we have an open subset $U\subset C$ such that $\pi:C_0:=\pi^{-1}(U)\rightarrow U$ is an isomorphism. We may regard $C_0$ as a section of $\pi$ defined over $U$. That is 
\begin{center}
$\xymatrix{
\mathbb{P}((\mathcal{O}_V^{\oplus r}\oplus L_V)|_U)\ar[r]\ar[d] &\mathbb{P}(\mathcal{O}_V^{\oplus r}\oplus L_V) \ar[d]\\
U \ar@/^1pc/[u]^{\sigma} \ar[r] & V
}$
\end{center}
The section $\sigma$ is defined by a quotient $\rho=(\rho_1,\rho_2):(\mathcal{O}_V^{\oplus r}\oplus L_V)|_U\twoheadrightarrow M$, with $M$ a line bundle on $U$. The morphism $\rho$ has a decomposition into $\rho_1:\mathcal{O}_V^{\oplus r}\rightarrow M$ and $\rho_2:L_V\rightarrow M$  . As $\psi_{|\xi|}\circ \sigma (U)=\{z\}$, we know that $M\cong \sigma^{\ast}(\xi|_{\mathbb{P}((\mathcal{O}_V^{\oplus r}\oplus L_V)|_U)})$ is trivial. As $h^0(\mathrm{Hom}_{\mathcal{O}_U}(L_V|_U,\mathcal{O}_U))=h^0(U,L_V^{\vee}|_U)=0$, we have that $\rho_2=0$. Hence the quotient is given by $\rho_1:\mathcal{O}_U^{\oplus n-1}\rightarrow \mathcal{O}_U$. Hence  $C_0=U\subset \mathbb{P}(\mathcal{O}_V^{\oplus r})$ and $C=\overline{C_0}\subset \mathbb{P}(\mathcal{O}_V^{\oplus r})$.

As $V=\mathbb{P}(L_V)$ is smooth, we have the short exact sequence
\begin{center}
$0\rightarrow T_{\mathbb{P}(L_V)}\rightarrow T_{\mathbb{P}(\mathcal{O}_V^{\oplus r}\oplus L_V)}|_{\mathbb{P}(L_V)}\rightarrow N_{\mathbb{P}(L_V)/\mathbb{P}(\mathcal{O}_V^{\oplus r}\oplus L_V)}\rightarrow 0$.
\end{center}
We have thus 
\begin{equation}
\label{eq1}
\omega_{\mathbb{P}(\mathcal{O}_V^{\oplus r}\oplus L_V)}^{\vee}|_{\mathbb{P}(L_V)}=\omega_{\mathbb{P}(L_V)}^{\vee}\otimes \wedge^{r}N_{\mathbb{P}(L_V)/\mathbb{P}(\mathcal{O}_V^{\oplus r}\oplus L_V)}.
\end{equation}
The canonical bundle formula gives us
\begin{equation*}
\omega_{\mathbb{P}(\mathcal{O}_V^{\oplus r}\oplus L_V)}=\pi^{\ast}(\omega_V\otimes L_V)\otimes \xi^{\otimes-(r+1)}.
\end{equation*}
With $\xi|_V=L_V$, we know that $\omega_{\mathbb{P}(\mathcal{O}_V^{\oplus r}\oplus L_V)}|_V=\omega_V\otimes L_V^{\otimes -r}$. Thus \cref{eq1} gives 
\begin{equation*}
\wedge^{r}N_{\mathbb{P}(L_V)/\mathbb{P}(\mathcal{O}_V^{\oplus r}\oplus L_V)}=L_V^{\otimes r}.
\end{equation*} 
As $\mathbb{P}(L_V)$ is disjoint from the singular locus $\mathbb{P}^{r-1}\subset X$, we also have the exact sequence 
\begin{center}
$0\rightarrow T_{\mathbb{P}(L_V)}\rightarrow T_X|_{\mathbb{P}(L_V)}\rightarrow N_{\mathbb{P}(L_V)/X}\rightarrow 0$.
\end{center}
Hence 
\begin{equation*}
\omega_X^{\vee}|_{\mathbb{P}(L_V)}=\omega_{\mathbb{P}(L_V)}^{\vee}\otimes \wedge^{r}N_{\mathbb{P}(L_V)/X}.
\end{equation*}
Note $N_{\mathbb{P}(L_V)/X}=N_{\mathbb{P}(L_V)/\mathbb{P}(\mathcal{O}_V^{\oplus r}\oplus L_V)}$. Hence $\omega_X|_V=\omega_V\otimes L^{\otimes-r}$.
Then we have 
\begin{center}
$\omega_X\otimes L^{\otimes n}|_V=\omega_V\otimes L^{\otimes (n-r)}$.
\end{center}
Hence the divisor $K_V+\mathrm{dim(V)}L_V$ is ample. 

If $\mathrm{dim(V)}\geq 2$, apply \citep[Theorem~5.10]{fujita_1990} again for $(V,L_V)$. We know that $(V,L_V)$ is one of the following:
\begin{itemize}
\item $(\mathbb{P}^{\mathrm{dim}(V)}, \mathcal{O}_{\mathbb{P}^{\mathrm{dim}(V)}}(1))$; or
\item $(Q,\mathcal{O}_Q(1))$, where $Q\subset \mathbb{P}^{\dim(V)+1}$ is a hyperquadric; or
\item $(\mathbb{P}(\mathcal{E}),\mathcal{O}_{\mathbb{P}(\mathcal{E})}(1))$ where $\mathcal{E}$ is an ample vector bundle of rank $\dim(V)$ over $\mathbb{P}^1$; or
\item $(\mathbb{P}^2,\mathcal{O}_{\mathbb{P}}^2(2))$
\end{itemize}
Suppose first that $\dim(V)=2$. If $(V,L)$ is $(\mathbb{P}^{1}, \mathcal{O}_{\mathbb{P}^{1}}(1))$ or $(Q,\mathcal{O}_{\mathbb{P}^{3}}(1))$, the divisor $K_V+2L_V$ will not be ample. If $(V,L)$ is a $\mathbb{P}^{1}$-bundle over $\mathbb{P}^{1}$, then $K_V+2L_V$ is trivial on each fiber, contradicting to the fact that $K_V+2L_V$ is ample. Hence we have $(V,L)\cong (\mathbb{P}^2,\mathcal{O}_{\mathbb{P}}^2(2))$.

If $\mathrm{dim}(V)=1$, we have that $(V,L_V)\cong (\mathbb{P}^1,\mathcal{O}_{\mathbb{P}^1}(a))$ with $a\geq 3$. By  the following \cref{example1} we know that for $n\geq 2$, a generalized cone $C_n(\mathbb{P}^1,\mathcal{O}_{\mathbb{P}^1}(a))$ has singularities worse than canonical.

Hence when $(X,L)$ is a generalized cone, we have that
$(X,L) \cong C_n(\mathbb{P}^2,\mathcal{O}_{\mathbb{P}}^2(2))$.
\end{proof}

We give some characterizations of generalized cones over $\mathbb{P}^1$.
\begin{example}
\label{example1}
Let $(X,L)=C_n(\mathbb{P}^1,\mathcal{O}_{\mathbb{P}^1}(a))$ be a generalized cone with $a \geq 3$ and $n\geq 2$. We have 
\begin{enumerate}
\item $X$ has klt singularities and $X$ is not canonical;
\item the nefvalue of $L$ is $n-\frac{a-2}{a}$;
\item $K_X+(n-1)L$ is not pseudo-effective.
\end{enumerate}
\end{example}

\begin{example}
\label{nonexample}
Let $(X,L)=C_n(\mathbb{P}^1,\mathcal{O}_{\mathbb{P}^1}(a))$ be a generalized cone with $a \leq 2$ and $n\geq 2$. Then $K_X+(n-1)L\in \mathrm{Pseff}(X)$.
\end{example}
For the proof of the above two examples, see \citep[Lemma~3.11 and 3.12]{Liu-thesis}.

\subsection*{Proof of \cref{Canonicalcase}}
The proof is by combining all the precedent results.
\begin{proof}
By \cref{r}, we have that $(X,L)\sim_{\mathrm{bir}} (X',L')$, where $X'$ is a normal projective variety with canonical $\mathbb{Q}$-factorial singularity, $K_{X'}+(n-1)L'\notin \mathrm{Pseff}(X')$ and 
\begin{enumerate}
\item[(0-i)] Either $\tau(L')$ is finite;
\item[(0-ii)] or there is a Mori fiber space structure $\phi:X'\rightarrow W$ and a rational number $\tau>(n-1)$ such that $L'$ is $\phi$-ample and $K_{X'}+\tau L'\sim_{\mathbb{Q},\phi}0$.
\end{enumerate}
In the first case, we have that $r(L')=\frac{1}{\tau(L')}>0$, hence by Kawamata rationality theorem there exists an $K_X'$-negative extremal ray $R_0=\mathbb{R}_{\geq 0}[C_0]$ such that $(r(L')K_{X'}+L')\cdot C_0=0$. Hence $L'\cdot C_0>0 $. In the second case, take $R_0=\mathbb{R}_{\geq 0}[C_0]$ be the extremal ray associated to $\phi$. Then $L'\cdot C_0>0$. 

Applying \cref{cl1} on $(X',L')$, we get that $(X',L')$ is the one of the following
\begin{enumerate}
\item $(X',L')\cong (\mathbb{P}^n,\mathcal{O}_{\mathbb{P}^n}(1))$, and $\tau=n+1$;
\item[(2-i)] $(X',L')$ is isomorphic to a $(\mathbb{P}^{n-1},\mathcal{O}_{\mathbb{P}^{n-1}}(1))$-bundle over a smooth curve $C$ and $\tau=n$;
\item[(\protect{$\ast$})] $\Delta(X',L')=0$, $K_{X'}+\tau L'\equiv_{\mathrm{num}} \mathcal{O}_X$ and $n-1<\tau\leq n$.
\end{enumerate}
If we are in case $(\ast)$, apply \cref{cl2}. We have a birational morphism $\mu:X'\rightarrow X''$ such that 
\begin{enumerate}
\item[(a)] $X''$ has canonical singularities, $\mu^{\ast}(K_{X''})=K_{X''}$ ;
\item[(b)] There exists an ample line bundle $L''$ on $X''$ such that $\mu^{\ast}(L'')=L'$;
\item[(c)] $\Delta(X'',L'')=0$ and $K_{X''}+\tau L''\equiv_{\mathrm{num}} \mathcal{O}_{X''}$.
\end{enumerate}

In particular we have that $(X',L')\sim_{\mathrm{bir}} (X'',L'')$. Now apply \cref{cl3} to $(X'',L'')$. We have that $(X'',L'')$ is isomorphic to the following pair:
\begin{enumerate}
\item[(3)] $(X'',L'')\cong (Q,\mathcal{O}_{\mathbb{P}^{n+1}}(1))$, where $Q\subset \mathbb{Q}^{n+1}$ is a hyperquadric;
\item[(2-ii)] $(X'',L'')$ is a $\mathbb{P}^{n-1}$-bundle over $\mathbb{P}^1$ and $L$ restricted to each fiber is $\mathcal{O}_{P^{n-1}}(1)$;
\item[(4)] $(X'',L'')\cong (\mathbb{P}^2,\mathcal{O}_{\mathbb{P}}^2(2))$;
\item[(5)] $(X'',L'')\cong C_n(\mathbb{P}^2,\mathcal{O}_{\mathbb{P}}^2(2))$ is a generalized cone over $(\mathbb{P}^2,\mathcal{O}_{\mathbb{P}}^2(2))$
\end{enumerate}
Thus we get the list stated in \cref{Canonicalcase} 
\end{proof}

\section{Normal polarized varieties}
\label{npv}
With the help of canonical modification \citep[Theorem~1.31]{kol13}, we can give a classification theorem for normal polarized varieties with $\mathbb{Q}$-Gorenstein singularities.
\subsection*{Proof of \cref{Cl}}

\begin{proof}
Apply \citep[Theorem~1.31]{kol13} to the pair $(X,0)$. We get the canonical modification $f:X'\rightarrow X$ with $K_{X'}$ being $f$-ample. We take a further step, taking a small $\mathbb{Q}$-factorial modification $g:Y \rightarrow X'$ of $X'$ (\emph{cf.} \citep[Corollary~1.37]{kol13}). We denote the composition $g\circ f$ by $\mu$. As $g$ is small, we have that $K_{Y}=g^{\ast}(K_{X'})$ is $\mu$-nef. Note that $\mu|_{\mu^{-1}(X_{\mathrm{reg}})}:\mu^{-1}(X_{\mathrm{reg}})\rightarrow X_{\mathrm{reg}}$ is an isomorphism.

We have that $\mu_{\ast}(\omega_Y)|_{X_{\reg}}\cong \omega_X|_{X_{\mathrm{reg}}}$ for the canonical sheaves $\omega_Y=\mathcal{O}_Y(K_Y)$ and $\omega_X=\mathcal{O}_X(K_X)$. Note that $\mu_{\ast}(\omega_Y)$ is torsion-free, so we have an injection $\mu_{\ast}(K_Y)\rightarrowtail K_X$. By the projection formula we have an injection 
\begin{center}
$\mathcal{O}_X(\mu_{\ast}(K_Y+(n-1)\mu^{\ast}L))\rightarrowtail \mathcal{O}_X(K_X+(n-1)L)$.
\end{center}
As $K_X+(n-1)L$ is not pseudo-effective, we know that neither is $K_Y+(n-1)\mu^{\ast}(L)$. 
We set $\mu^{\ast}(L)=M$. As $M$ is nef and big, we know that $M\in \mathrm{Pseff}(Y)$. Note that $K_Y$ is not pseudo-effective, hence it is not nef.

Let $R=\mathbb{R}_{\geq 0}[C]$ be a $K_Y$-negative extremal ray, with $C\subset Y$ a rational curve. We have that $K_Y\cdot C<0$. As $K_Y$ is $\mu$-nef, we know that $C$ is not contracted by $\mu$. Hence $\mu(C)\subset X$ has dimension $1$. The intersection number $M\cdot C=\mathrm{deg}(C/\mu(C)) L\cdot \mu(C)$ is positive, since $L$ is ample. Thus for any $K_Y$-negative extremal ray $R$, one has $M\cdot R>0$. By \cref{r}, we obtain that $\mathrm{r}(M)>0$ and $\tau(M)>n-1$. 
By \cref{cl1} applied to $(Y,M)$, we have one of the following cases:
\begin{enumerate}
\item[(i)] $(Y,M)\cong (\mathbb{P}^n,\mathcal{O}_{\mathbb{P}^n}(1))$, and $\tau=n+1$, or
\item[(ii)] $(Y,M)$ is isomorphic to a $(\mathbb{P}^{n-1},\mathcal{O}_{\mathbb{P}^{n-1}}(1))$-bundle over a smooth curve $C$ and $\tau=n$, or
\item[(iii)]$\Delta(Y,M)=0$, $K_Y+\tau M\sim_{\mathbb{Q}} \mathcal{O}_Y$ and $n-1<\tau\leq n$.
\end{enumerate}

In case (i), we have a birational morphism $\mu:\mathbb{P}^n\rightarrow X$ with $\mu^{\ast}(L)=\mathcal{O}_{\mathbb{P}^n}(1)$. We have that $\NE(\mathbb{P}^n/X)=0$ since both $L$ and $\mathcal{O}_{\mathbb{P}^n}(1)$ are ample. By \citep[Proposition~1.14]{deb}, the morphism $\mu$ is an isomorphism. We have case ($1$) in \cref{Cl}.

In case (ii), we have a birational morphism $\mu:\mathbb{P}(\mathcal{V})\rightarrow X$, such that $K_{\mathbb{P}(\mathcal{V})}$ is $\mu$-nef. We denote by $\xi$ the pull-back $\mathcal{O}_{\mathbb{P}(\mathcal{V})}(1)=\mu^{\ast}(L)$. We know that $\xi$ is nef and big. 

We first note that $\xi$ is ample if and only if $\mu$ is an isomorphism. In fact, if $\mu$ is an  isomorphism, then we have $\xi$ is ample. Conversely, if $\xi$ is ample, we have $\NE(\mathbb{P}(\mathcal{V})/X)=0$ and hence $\mu$ is an isomorphism. In this case, we have that 
\begin{center}
$K_{\mathbb{P}(\mathcal{V})}+(n-1)\xi=\pi^{\ast}(K_C+\mathrm{det}\mathcal{V})-\xi$
\end{center}
is not pseudo-effective.  %by \citep[Page~450]{Fulger} $\mathrm{Pseff}(\mathbb{P}(\mathcal{V}))=\mathbb{R}_{\geq 0}f+\mathbb{R}_{\geq 0}(\xi+\nu^{(1)}f)$ for some $\nu^{(1)}\in \mathbb{Q}$.
In fact, the general fiber $f$ is from a covering family and we have that $K_{\mathbb{P}(\mathcal{V})}+(n-1)\xi|_f=\mathcal{O}_f(-1)$. Hence by the BDPP theorem (\emph{cf.} \citep[Theorem~11.4.19]{Laz04}), we know that $K_{\mathbb{P}(\mathcal{V})}+(n-1)\xi$ is not pseudo-effective. Thus we get (\textit{2,i}).

 Now suppose that $\xi$ is not ample. Then $\mu$ is not an ismorphism. We have the following diagram: 
\begin{center}
$\xymatrix{(\mathbb{P}(\mathcal{V}),\xi)\ar[d]_{\pi} \ar[r]^{\mu} & (X,L)\\
C}$.
\end{center}
We know that $\rho(\mathbb{P}(\mathcal{V}))=2$. As $\mu \neq \pi$, we have that
\begin{center}
$\overline{\mathrm{NE}}(\mathbb{P}(\mathcal{V}))=\mathrm{NE}(\pi)+\mathrm{NE}(\mu)$.
\end{center}
We denote a general fiber of $\pi$ by $f$. By \citep[Page~450]{Fulger}, we know that $\overline{\mathrm{NE}}(\mathbb{P}(\mathcal{V}))$ has as extremal rays $\mathbb{R}_{\geq 0}\xi^{n-2}f$ and  $\mathbb{R}_{\geq 0}(\xi^{n-1}+\nu^{(n-1)} \xi^{n-2}f)$ for some $\nu^{(n-1)} \in \mathbb{Q}$. Note that $\mathbb{P}^1=\xi^{n-2}f$ is contracted by $\pi$. Hence $\mathrm{NE}(\pi)=\mathbb{R}_{\geq 0}\xi^{n-2}f$. We have $K_{\mathbb{P}(\mathcal{V})}=\pi^{\ast}(K_C+\mathrm{det}(\mathcal{V}))-n\xi$. Hence $K_{\mathbb{P}(\mathcal{V})}\cdot \xi^{n-2}f=-n$. Thus $\pi$ is the Mori contraction associated to the extremal ray $\mathbb{R}_{\geq 0}\xi^{n-2}f$.
As $\NE(\mu)$ is an extremal ray, we know that $\mu$ is an extremal contraction. By \citep[Proposition~2.5.]{KM98}, we know that $\mu$ is either small or divisorial.

If $\mu$ is small, we have that $K_{\mathbb{P}(\mathcal{V})}=\mu^{\ast}(K_X)$. As $\rho(X)=1$, we have that $K_X\equiv_{\mathrm{num}}mL$ for some $m\in \mathbb{Q}$. Hence $K_{\mathbb{P}(\mathcal{V})}\equiv_{\mathrm{num}} m\xi$. We have that
\begin{center}
$m= m\xi\cdot\xi^{n-2}f=K_{\mathbb{P}(\mathcal{V})}\cdot \xi^{n-2}f  =-n$.
\end{center}
 Thus we get that $K_X+nL\num\mathcal{O}_X$. By \cref{KOG}, we have that $(X,L)\cong (Q,\mathcal{O}_Q(1))$ where $Q\subset \mathbb{P}^{n+1}$ is a hyperquadric. Hence we are in case (3) of \cref{Cl}.

If $\mu$ is divisorial, we denote the exceptional divisor by $E=\mathrm{exc}(\mu)$. Note that $\mathcal{V}$ is nef, since $\mu^{\ast}(L)=\xi=\mathcal{O}_{\mathcal{V}}(1)$ is nef. We have a unique exact sequence of locally free sheaves:
\begin{center}
$0\rightarrow \mathcal{A} \rightarrow \mathcal{V}\rightarrow \mathcal{Q}\rightarrow 0$.
\end{center}
with $\mathcal{A}$ being an ample vector bundle and $\mathcal{Q}$ being numerically flat. If $l\subset \mathbb{P}(\mathcal{V})$ is a curve such that $\xi \cdot l=0$, we have that $l\subset \mathbb{P}(\mathcal{Q})$. Thus we have that $E\subset \mathbb{P}(\mathcal{Q})$. In particular, $\mathrm{rk}(\mathcal{Q})=n-1$ and $E= \mathbb{P}(\mathcal{Q})$. We denote the bundle morphism by $\pi': \mathbb{P}(\mathcal{Q})\rightarrow C$. Now we compute $E|_E$:
\begin{align*}
E|_E &= (K_{\mathbb{P}(\mathcal{Q})}-K_{\mathbb{P}(\mathcal{V})})|_E\\
     &=\pi'^{\ast}(K_C+\mathrm{det}\mathcal{Q})-(n-1)\xi|_E-(\pi^{\ast}(K_C+\mathrm{det}\mathcal{V})-n\xi)|_E\\
     &=\pi'^{\ast}(\mathrm{det}\mathcal{Q}-\mathrm{det}\mathcal{V})+\xi|_E\\
     &=\pi'^{\ast}(-\mathcal{A})+\xi|_E .
\end{align*}
Take a rational curve $l$ that is in the fiber of $\pi'$. We have  that $E\cdot l=E|_E\cdot l=1$. Now write $K_{\mathbb{P}(\mathcal{V})}=\mu^{\ast}(K_X)+\lambda E$. As $\rho(X)=1$, we have that $K_X\equiv_{\mathrm{num}}m L$ for some $m\in \mathbb{Q}$. As $K_X+(n-1)L\equiv_{\mathrm{num}} (m+n-1)L\notin \mathrm{Pseff}(X)$, we have that $m+n< 1$. Intersecting with $l$, we get that
\begin{center}
$-n=K_{\mathbb{P}(\mathcal{V})}\cdot l =(\mu^{\ast}(K_X)+\lambda E)\cdot l=(m\xi+\lambda E)\cdot l=m+\lambda$.
\end{center}
Hence $\lambda=-m-n>-1$ and $X$ has klt singularities. A $\pi'$-fiber is isomorphic to $\mathbb{P}^{n-2}$ and is mapped isomorphically onto its image by $\mu$. Hence each non-trivial $\mu$-fiber has dimension $1$.
 As $X$ has klt singularities, \textit{a fortiori} $(X,0)$ is dlt. Applying \citep[Corollary~1.5-(1)]{HM07} to the birational morphism $\mu$, each $\mu$-fiber is rationally chain connected. Hence a non trivial fiber has $\mathbb{P}^1$ as its normalization. We have thus a finite map $\pi'|_{\mathbb{P}^1}:\mathbb{P}^1\rightarrow C$. Thus $C=\mathbb{P}^1$ and $\mathcal{V}=\mathcal{O}_{\mathbb{P}^1}(a)\oplus \mathcal{O}_{\mathbb{P}^1}^{\oplus (n-1)}$. 
 
 Consider the morphism $\psi:\mathbb{P}(\mathcal{O}_{\mathbb{P}^1}(a)\oplus \mathcal{O}_{\mathbb{P}^1}^{\oplus (n-1)})\rightarrow C_n(\mathbb{P}^1,\mathcal{O}_{\mathbb{P}^1}(a))$. We know that $\psi$ does not contract the extremal ray $\NE(\pi)$. Hence $\NE(\psi)=\NE(\mu)$ and by \citep[Proposition~1.14]{deb} $X=C_n(\mathbb{P}^1,\mathcal{O}_{\mathbb{P}^1}(a))$. As $L$ and the restriction of $\mathcal{O}_{\mathbb{P}(H^0(\mathbb{P}^1,\mathcal{O}_{\mathbb{P}^1}(a)))}$ to $C_n(\mathbb{P}^1,\mathcal{O}_{\mathbb{P}^1}(a))$ agree outside a subscheme of codimension at least $2$, we have that $(X,L)=C_n(\mathbb{P}^1,\mathcal{O}_{\mathbb{P}^1}(a))$. As $K_X+(n-1)L\notin \mathrm{Pseff}(X)$, \cref{nonexample} implies $a\geq 3$. Now \cref{example1} shows that for all $a\geq 3$, the divisor  $K_X+(n-1)L$ is not pseudo-effective and $X$ is klt. Thus we get (\textit{2,ii}).

If we are in case (iii), apply \cref{cl2} to $(Y,M)$. We have a crepant resolution $\nu:Y\rightarrow Y_{\mathrm{can}}$ with an ample divisor $A$ on $Y_{\mathrm{can}}$ such that $\nu^{\ast}(A)=M$, the $\Delta$-genus satisfies $\Delta(Y_{\mathrm{can}},A)=0$ and $K_{Y_{\mathrm{can}}}+\tau A\equiv_{\mathrm{num}} \mathcal{O}_{Y_{can}}$. By \cref{cl3}, $(Y_{\mathrm{can}},A)$ is isomorphic to one of the following:
\begin{enumerate}
\item[(a)] $(Q,\mathcal{O}_{\mathbb{P}^{n+1}}(1))$, where $Q\subset \mathbb{Q}^{n+1}$ is a hyperquadric;
\item[(b)] a $\mathbb{P}^{n-1}$-bundle over $\mathbb{P}^1$ and $L$ restricted to each fiber is $\mathcal{O}_{\mathbb{P}^{n-1}}(1)$;
\item[(c)] $(\mathbb{P}^2,\mathcal{O}_{\mathbb{P}^2}(2))$;
\item[(d)]  a generalized cone $ C_n(\mathbb{P}^2,\mathcal{O}_{\mathbb{P}^2}(2))$ over $(\mathbb{P}^2,\mathcal{O}_{\mathbb{P}^2}(2))$.
\end{enumerate}

Case (b) is a special case of (ii) treated above. In case (a),(c) and (d), we have the following diagram 
\begin{center}
$\xymatrix{(Y,M)\ar[r]^{\mu} \ar[d]^{\nu}& (X,L)\\
(Y_{\mathrm{can}},A)\ar@{-->}[ur]_h}$,
\end{center}
where $h$ is a birational map $a$ $priori$ not necessarily defined on all $Y_{\mathrm{can}}$. We now show $h$ is indeed an isomorphism and $h^{\ast}(L)=A$. Let $C\subset Y$ be a curve. We have
\begin{equation*}
\nu^{\ast}(A)\cdot C=M\cdot C=\mu^{\ast}(L)\cdot C.
\end{equation*}
As $A$ and $L$ are both ample, we have that $\mathrm{NE}(\mu)= \mathrm{NE}(\nu)$. \citep[Proposition~1.14]{deb} implies that $h$ is an isomorphism. As $h^{\ast}(L)$ agrees with $A$ outside a subscheme of codimension at least $2$, we have that $h^{\ast}(L)=A$. Hence we get case $(3)$, $(4)$, $(5)$ in \cref{Cl}.
\end{proof}
 
 Using similar methods, we can classify log pairs $(X,\Delta)$ with $\Delta$ a reduced Weil divisor. 
 \subsection*{Proof of \cref{Clr}}
\begin{proof}

 We take a canonical modification of $X$ then take a small $\mathbb{Q}$-factorialization. We get a birational morphism $\mu:Y\rightarrow X$ such that $Y$ has $\mathbb{Q}$-factorial canonical singularities, $K_Y$ is $\mu$-nef and $\mu$ is isomorphic over regular points of $X$. Set $\Delta':=\mu^{-1}_{\ast}(\Delta)$. Then $\Delta'$ is a reduced divisor. Let $\omega_Y=\mathcal{O}_Y(K_Y)$ and $\omega_X=\mathcal{O}_X(K_X)$ be the canonical sheaves.  We know that 
\begin{equation*}
\mu_{\ast}(\omega_Y\otimes\mathcal{O}_Y(\Delta'))|_{X_{\reg}}\cong (\omega_X\otimes\mathcal{O}_X(\Delta))|_{X_{\reg}}.
\end{equation*}
 The sheaf  $\mu_{\ast}(\omega_Y\otimes\mathcal{O}_Y(\Delta'))$ is torsion-free, so we have an injection 
 \begin{equation*}
 \mu_{\ast}(\omega_Y\otimes\mathcal{O}_Y(\Delta'))\rightarrowtail \omega_X\otimes\mathcal{O}_X(\Delta) .
 \end{equation*}
Tensoring with $\mu^{\ast}(L^{\otimes n-1})$, we have an injection
\begin{center}
$\mu_{\ast}(\omega_Y\otimes\mathcal{O}_Y(\Delta')\otimes\mu^{\ast}(L^{\otimes n-1}))\rightarrowtail \omega_X\otimes\mathcal{O}_X(\Delta)\otimes L^{\otimes n-1}$.
\end{center}
As $(K_X+\Delta)+(n-1)L$ is not pseudo-effective,  neither is $(K_Y+\Delta')+(n-1)\mu^{\ast}(L)$. 
We set $\mu^{\ast}(L)=:M$. As $\Delta'$ is effective, the divisor $K_Y+(n-1)M$ is not pseudo-effective.

As $K_Y$ is $\mu$-nef and $M=\mu^{\ast}(L)$, for any $K_Y$-negative extremal ray $R$, we have that $M\cdot R>0$. Hence we can apply \cref{cl1} to $(Y,M)$ and get:

\begin{enumerate}
\item[(a)] $(Y,M)\cong (\mathbb{P}^n,\mathcal{O}_{\mathbb{P}^n}(1))$, and $\tau=n+1$;
\item[(b)] $(Y,M)$ is isomorphic to a $(\mathbb{P}^{n-1},\mathcal{O}_{\mathbb{P}^{n-1}}(1))$-bundle over a smooth curve $C$ and $\tau=n$;
\item[(c)] $\Delta(Y,M)=0$, $K_Y+\tau M\sim_{\mathbb{Q}} \mathcal{O}_Y$ and $n-1<\tau(M)\leq n$.
\end{enumerate}

If we are in case $(a)$,  the morphism $\mu$ is an isomorphism. The divisor $\Delta$ is given by $\mathcal{O}_{\mathbb{P}^n}(a)$ for some $a\geq 1$. We have that $K_X+(n-1)L+D=\mathcal{O}_{\mathbb{P}^n}(a-2)$. Hence the only possible choice is $a=1$ and $\Delta=D$ is a hyperplane. We are thus in case $(1)$ of \cref{Clr}.

If we are in case $(b)$, we have a diagram 
\begin{center}
$\xymatrix{(\mathbb{P}(\mathcal{V}),\xi)\ar[d]_{\pi} \ar[r]^\mu & (X,L)\\
C}$,
\end{center}
where $\xi=\mathcal{O}_{\mathbb{P}(\mathcal{V})}(1)$ and $K_{\mathbb{P}(\mathcal{V})}$ is $\mu$-nef.

 First we assume that $\mu$ is an isomorphism. In this case the vector bundle $\mathcal{V}$ is ample and $X=\mathbb{P}(\mathcal{V})$ is $\mathbb{Q}$-factorial. Let $F=\mathbb{P}^{n-1}$ be a general fiber of $\pi$ . Suppose that $\Delta|_F=\mathcal{O}_F(d)$ for some natural number $d\geq 0$. We have the following equality:
 \begin{center}
 $(K_{\mathbb{P}(\mathcal{V})}+\Delta+(n-1)\xi)|_F=(\pi^{\ast}(K_C+\mathrm{det}(\mathcal{V}))+\Delta-\xi)|_F=\mathcal{O}_F(d-1)$.
 \end{center}
 If $d=0$, let $D$ be a component of $\Delta$, then $D|_F=\mathcal{O}_F(0)$. We claim that $D$ is one of the general fiber. In fact, suppose by contradiction that there exists a general fiber $F$ such that $D\cap F\neq \emptyset$ and $D\nsubseteq F$. Then there will be a curve $l\subset F\setminus D$ such that $l\cap D\neq \emptyset$. Then we have that $D\cdot l>0$, a contradiction. Thus we have that $\Delta=\sum F_i$ is a finite sum of distinct general fibers.
  Let $l$ be a rational curve in $F$. We have that 
 \begin{equation*}
 (K_{\mathbb{P}(\mathcal{V})}+\Delta+(n-1)\xi)\cdot l=-1.
 \end{equation*}
Since $F$ is a member of a covering family, BDPP theorem (\emph{cf.} \citep[Theorem~11.4.19]{Laz04}) implies that $K_{\mathbb{P}(\mathcal{V})}+(n-1)\xi+\Delta$ is not pseudo-effective. We  are thus in case $(2.i)$ of \cref{Clr}.
 
  If $d>0$, let $D$ be a component of $\Delta$ such that $D|_F=\mathcal{O}_F(d')$ for some $d'>0$. By \cref{dimension-control} after the proof, we have that $n=\mathrm{dim}(\mathbb{P}(\mathcal{V}))=2$. We first show that $C=\mathbb{P}^1$. The non pseudo-effective divisor in question $K_X+D+(n-1)L$  thus becomes $K_{\mathbb{P}(\mathcal{V})}+D+\xi$. We have that $(K_{\mathbb{P}(\mathcal{V})}+D+\xi)|_F=\mathcal{O}_F(d'-1)$ which is nef.
As $K_{\mathbb{P}(\mathcal{V})}+D+\xi$ is not nef,  we note that there will be an extremal ray $R'$ which is not generated by the fiber of $\pi$, such that $(K_{\mathbb{P}(\mathcal{V})}+D+\xi)\cdot R'<0$. In particular, we have that $R'$ is $(K_{\mathbb{P}(\mathcal{V})}+D)$-negative. By the cone theorem, we know that $R'=\mathbb{R}_{\geq 0}[l]$ for a rational curve $l$. Note that $l$ maps finitely onto $C$. Hence we have $C\cong \mathbb{P}^1$.

Thus $(X,L)$ is a $\mathbb{P}^1$-bundle over $\mathbb{P}^1$ and  \cref{dim-2-class} implies that $(X,\Delta,L)$ is either in cases $(2.i),$ $(2.ii)$ of \cref{Clr} or $(X,L)$ is a hyperquadric of rank $4$, which will be dealt in the following case $(c1)$. 
 
 Assume, from now on, that $\mu$ is not an isomorphism. We know that 
\begin{equation*}
\NE(\mathbb{P}(\mathcal{V}))=\NE(\mu)+\NE(\pi).
\end{equation*}
 and $\pi$ contracts the extremal ray $\mathbb{R}_{\geq 0}\xi^{n-2}f$. The birational morphism  $\mu$ is either small or divisorial.

 If $\mu$ is small, by construction, we have that $K_{\mathbb{P}(\mathcal{V})}+\Delta'=\mu^{\ast}(K_X+\Delta)$. Let $F$ be a general fiber of $\pi$. We have that $\Delta'|_F=\mathcal{O}_F(d)$ for some integer $d\geq 0$. As $K_X+\Delta$ is $\mathbb{Q}$-Cartier and $\rho(X)=1$, we have that $K_X+\Delta\num mL$ for some $m\in \mathbb{Q}$. Hence $K_{\mathbb{P}(\mathcal{V})}+\Delta'\num\mu^{\ast}mL$. Intersect with $\xi^{n-2}f$. We get that $-n+d=m$. Hence $K_X+\Delta+(n-1)L\equiv_{\mathrm{num}}(d-1)L$. Thus $d=0$.Hence $d=0$. If we write $\Delta'=\sum D'_i$ with $D'_i$ distinct prime divisors. We have that $D_i'|_F=\mathcal{O}_F(0)$. Thus the $D_i'$'s are distinct general fibers. As $D_i'=\mu_{\ast}^{-1}(D_i)$ by definition, we get that $D_i'\rightarrow D_i=\mu(D'_i)$ has degree $1$. Thus we are in case $(2.i)$ of \cref{Clr}.   
 
 If $\mu$ is divisorial, we denote the exceptional divisor by $E=\mathrm{exc}(\mu)$. \citep[Proposition~3.36.]{KM98} implies that $X$ is $\mathbb{Q}$-factorial. In particular $K_X$ is $\mathbb{Q}$-Cartier. 
We have a unique exact sequence of locally free sheaves:
\begin{center}
$0\rightarrow \mathcal{A} \rightarrow \mathcal{V}\rightarrow \mathcal{Q}\rightarrow 0$
\end{center}
with $\mathcal{A}$ is an ample vector bundle and $\mathcal{Q}$ is numerically flat. And we know that $E=\mathbb{P}(\mathcal{Q})$ and $E\cdot \xi^{n-2}f=1$. Let $F$ be a general fiber of $\pi$. There exists a $d\geq 0$ such that $\Delta'|_F=\mathcal{O}_F(d)$ . As $K_X+\Delta$ is $\mathbb{Q}$-Cartier and $\rho(X)=1$, there exists an $m\in \mathbb{Q}$ such that $K_X+\Delta\equiv_{\mathrm{num}}mL$. Then $K_X+\Delta+(n-1)L\equiv_{\mathrm{num}}(m+n-1)L\notin\mathrm{Pseff}(X)$. Hence $m+n<1$. We now have 
\begin{center}
$K_{\mathbb{P}(\mathcal{V})}+\Delta'=\mu^{\ast}(K_X+\Delta)+\lambda E$.
\end{center}
Intersect both sides with $\xi^{n-2}f$. We get that $-n+d=m+\lambda$. Since $-(m+n)>-1$, we have that $\lambda\geq -1+d$.

 Now we claim that $d=0$. Suppose by contradiction that $d\geq 1$. Then we have that $(K_{\mathbb{P}(\mathcal{V})}+\Delta'+(n-1)\xi)|_F=\mathcal{O}_F(d-1)$. As $K_{\mathbb{P}(\mathcal{V})}+\Delta'+(n-1)\xi$ is not nef, we know that $\mathrm{NE}(\mu)$ is an $(K_{\mathbb{P}(\mathcal{V})}+\Delta')$-negative extremal ray. 
Note that $(\mathbb{P}(\mathcal{V}),\Delta')$ is log canonical. By the cone theorem, there is a rational curve $l$ whose class $[l]$ is in $\NE(\mu)$. As $l$ maps finitely onto $C$, we know that $C\cong \mathbb{P}^1$. Hence $(X,L)=C_n(\mathbb{P}^1,\mathcal{O}_{\mathbb{P}^1}(a))$. As $K_X+(n-1)L$ is not pseudo-effective, \cref{nonexample} implies that $a\geq 3$. \cref{example1} implies that $K_X\equiv_{\mathrm{num}} (-n+\dfrac{a-2}{a})L$.
Suppose that $\Delta\equiv_{\mathrm{num}} m_2 L$ for some $m_2\in \mathbb{Q}^{+}$. For $\mathbb{P}^1\subset F$ mapped isomorphic to its image, we have that $m_2=m_2 \xi\cdot \mathbb{P}^1=\mu^{\ast}(\Delta)\cdot \mathbb{P}^1=\Delta\cdot \mu(\mathbb{P}^1)\in \mathbb{N}$. Hence $m_2\geq 1$ and $K_X+\Delta+(n-1)L=\dfrac{a-2}{a}L\in \mathrm{Pseff}(X)$, a contradiction. This proves the claim.

Write $\Delta=\sum D_i$ and $\Delta'=\sum D_i'$. Then each $D_i'$ is a general fiber. As $D_i'=\mu_{\ast}^{-1}(D_i)$ by definition, we get that $D_i'\rightarrow D_i$ has degree $1$. Thus we have that $D_i\cong \mu(\mathbb{P}^{n-1})$ are the images of distinct general fibers of $\pi$ and we are in case $(2.i)$ of \cref{Clr}.

If we are in case $(c)$, apply \cref{cl2} to $(Y,M)$. We have a crepant resolution $\nu:Y\rightarrow Y_{\mathrm{can}}$ with an ample divisor $A$ on $Y_{\mathrm{can}}$ such that $\nu^{\ast}(A)=M$, the $\Delta$-genus $\Delta(Y_{\mathrm{can}},A)=0$ and $K_{Y_{\mathrm{can}}}+\tau A\equiv_{\mathrm{num}} \mathcal{O}_{Y_{can}}$. By \cref{cl3}, we have one of the following cases:

\begin{enumerate}
\item[(c1)] $(Y_{\mathrm{can}},A)\cong (Q,\mathcal{O}_{\mathbb{P}^{n+1}}(1))$, where $Q\subset \mathbb{Q}^{n+1}$ is a hyperquadric;

\item[(c2)] $(Y_{\mathrm{can}},A)$ is a $\mathbb{P}^{n-1}$-bundle over $\mathbb{P}^1$ and the restriction of $L$ to each fiber is $\mathcal{O}_{\mathbb{P}^{n-1}}(1)$;

\item[(c3)] $(Y_{\mathrm{can}},A)\cong (\mathbb{P}^2,\mathcal{O}_{\mathbb{P}^2}(2))$;

\item[(c4)] $(Y_{\mathrm{can}},A)\cong C_n(\mathbb{P}^2,\mathcal{O}_{\mathbb{P}^2}(2))$ is a generalized cone over $(\mathbb{P}^2,\mathcal{O}_{\mathbb{P}^2}(2))$.
\end{enumerate}

We have the following diagram
\begin{center}
$\xymatrix{(Y,M)\ar[r]^{\mu} \ar[d]^{\nu}& (X,L)\\
(Y_{\mathrm{can}},A)\ar[ur]_h}$,
\end{center}
such that $h$ is an isomorphism and $\mu^{\ast}(L)=M=\nu^{\ast}(A)$ with $(Y_{\mathrm{can}},A)$ being one of the above four pairs.

In case $(c1)$, after an automorphism of $\mathbb{P}^{n+1}=\mathrm{Proj}(\mathbb{C}[x_0,...x_{n+1}])$, the hyperquadric $Q$ is given by the homogeneous ideal $I_r=(\sum_{0\leq i \leq r} x_i^2)\subset \mathbb{C}[x_0,...x_{n+1}]$ for some $r\geq 2$. By \citep[Exercise~II.6.5]{har78}, the class group $\mathrm{Cl}(Q)$ of $Q$ is the following:
\begin{enumerate}
\item[•] When $r=2$, $\frac{1}{2} [\mathcal{O}_Q(1)]$ is an integral divisor and $\mathrm{Cl}(Q)=\mathbb{Z}\cdot\frac{1}{2} [\mathcal{O}_Q(1)]$. Suppose that $\Delta=k\cdot \frac{1}{2}[\mathcal{O}_Q(1)]$. Write $\Delta=\sum D_i$. Then each . Then
\begin{center}
$K_X+(n-1)L+\Delta=(\frac{k}{2}-1)\mathcal{O}_Q(1)$.
\end{center}
It is not pseudo-effective if and only if $k=1$. Thus $\Delta=D$ is irreducible and is numerically equivalent to a hyperplane $\mathbb{P}^{n-1}$ in $Q$.  We are thus in case $(3.i)$ of \cref{Clr}.

\item[•] When $r=3$, $\mathrm{Cl}(Q)\cong \mathbb{Z}\oplus\mathbb{Z}$. Note that here we can write
\begin{center}
$Q=\mathrm{Proj}\left(\dfrac{\mathbb{C}[x_0, \ldots, x_{n+1}]}{(x_0 x_1-x_2 x_3)}\right)$,
\end{center}
which is a cone of vertex $\mathbb{P}^{n-3}=\left\lbrace x_1=x_2=x_3=0\right\rbrace \subset \mathbb{P}^{n+1}$ with base $\mathbb{P}^1 \times \mathbb{P}^1 \subset \mathbb{P}^3=\left\lbrace x_4=\cdots=x_{n+1}=0 \right\rbrace\subset \mathbb{P}^{n+1}$ (\emph{cf.} \citep[Exercise~I.5.12.(d)]{har78}). If we consider the inclusions $\mathbb{P}^3\subset\mathbb{P}^4\subset\cdots\subset \mathbb{P}^n\subset \mathbb{P}^{n+1}$, then $Q$ is also obtained by taking projective cone in the sense of \citep[Exercise~I.2.10]{har78} of $\mathbb{P}^1 \times \mathbb{P}^1 \subset \mathbb{P}^3$ successively. By \citep[Exercise~II.6.3.(a)]{har78}, we know $\mathrm{Cl}(\mathbb{P}^1 \times \mathbb{P}^1)\cong \mathrm{Cl}(Q)$.  For a hyperplane $H\subset\mathbb{P}^{n+1}$, $H\cap Q$ has type $(1,1)$. The cone over $\mathbb{P}^{1}\times \mathrm{pt}$ has type $(1,0)$ and the cone over $\mathrm{pt}\times \mathbb{P}^{1}$ has type $(0,1)$. Thus $\Delta$ has type $(1,0)$ or type $(0,1)$ and is irreducible. We are thus in case $(3.ii)$ of \cref{Clr}.

\item[•] When $r\geq 4$, $\mathrm{Cl}(Q)=\mathbb{Z}\cdot  [\mathcal{O}_Q(1)]$. Hence $\Delta= d[\mathcal{O}_Q(1)]$, and 
\begin{center}
$K_X+(n-1)L+\Delta\equiv_{\mathrm{num}}\mathcal{O}_Q(d-1)$
\end{center}
 is pseudo-effective. Thus this situation is excluded.

\end{enumerate} 

 The case $(c2)$ is treated in case $(b)$. The case $(c3)$ does not happen.
  
In case $(c4)$, we consider the following diagram
\begin{center}
$\xymatrix
{
E=\mathbb{P}^2\times\mathbb{P}^{n-3} \ar[r]^{\mathrm{pr}_2}\ar@{^{(}->}[d]^i & \mathbb{P}^{n-3}  \ar@{^{(}->}[d] \\ 
T=\mathbb{P}(\mathcal{O}^{\oplus n-2}\oplus \mathcal{O}(2))\ar[r]^-{\psi_{|\xi|}}\ar[d]^{\pi} & C_n(\mathbb{P}^2,\mathcal{O}(2))=X
\\ \mathbb{P}^2
}.$
\end{center}
Since $T$ is a projective bundle, by \citep[Theorem 3.3.(b)]{ful98} we have that
\begin{equation*}
\mathrm{Cl}(T)=\mathbb{Z} [\pi^{\ast}(\mathcal{O}_{\mathbb{P}^2}(1))] \oplus \mathbb{Z} [\xi].
\end{equation*} 
On the other hand, since $E=\mathrm{exc}(\psi_{|\xi|})$ is contracted, we know that the homomorphism $(\psi_{|\xi|})_{\ast}:\mathrm{Cl}(T)\rightarrow \mathrm{Cl}(X)$ is surjective and $\mathrm{rk}(\mathrm{Cl}(X))=1$. We have that $\psi_{|\xi|}^{\ast}(L)=\xi$. Thus $(\psi_{|\xi|})_{\ast}([\xi])=[L]\neq 0$. 
To determine $\mathrm{Cl}(X)$, one just need to know the image $(\psi_{|\xi|})_{\ast}\pi^{\ast}([\mathcal{O}_{\mathbb{P}^2}(1)])$. Let $H$ be a Weil divisor on $T$ such that $\mathcal{O}_T(H)=\pi^{\ast}(\mathcal{O}_{\mathbb{P}^2}(1))$. For example, we can take $H$ to be $\pi^{-1}(l)$ where $l\subset \mathbb{P}^{2}$ is a linear subspace. 
Then it's easy to see that $H\neq E$. Set $G:=(\psi_{|\xi|})_{\ast}H$. As $L$ is ample, the class $[L]$ is non-zero in $\mathrm{Cl}(X)\otimes \mathbb{Q}$. Take $m\in \mathbb{Q}$ such that $[G]=m[L]$ in $\mathrm{Cl(X)}\otimes \mathbb{Q}$. We have that
\begin{equation}
\label{Ediv}
\psi_{|\xi|}^{\ast}(G)\sim_{\mathbb{Q}}(\psi_{|\xi|})^{-1}_{\ast}(G)+a E,
\end{equation}
with $(\psi_{|\xi|})^{-1}_{\ast}(G)=H$.
By the canonical bundle formula, we have that
\begin{align*}
K_T &=\pi^{\ast}(\mathcal{O}_{\mathbb{P}^2}(-1))-(n-1)\xi \text{ } \text{ and}\\
K_E  &=\mathrm{pr}_1^{\ast}(\mathcal{O}_{\mathbb{P}^2}(-2))-(n-2)\xi|_E .
\end{align*}
Hence we have that
\begin{equation*}
\mathcal{O}_{E}(E)=\mathrm{pr}_1^{\ast}(\mathcal{O}_{\mathbb{P}^2}(-2))\otimes \mathrm{pr}_2^{\ast}(\mathcal{O}_{\mathbb{P}^{n-3}}(1)).
\end{equation*}
Let $C_1=\mathbb{P}^1\times \{\mathrm{pt}\}\subset E$. Then $E\cdot C_1=-2$. We intersect both sides of \cref{Ediv} with $C_1$. As $(\psi_{|\xi|})_{\ast}(C_1)=0$, by the projection  we get that $(\psi_{|\xi|})^{\ast}(G)\cdot C_1=0$. By applying the projection formula to the morphism $\pi|_H: H\rightarrow \mathbb{P}^2$, we get that $H\cdot C_1=1$. Hence $a=\frac{1}{2}$. Thus we have that
\begin{equation}
\label{eqxi}
m[\xi]=\pi^{\ast}[\mathcal{O}_{\mathbb{P}^2}(1)]+\frac{1}{2}E.
\end{equation}
Let $F=\mathbb{P}^{n-2}$ be a fiber of $\pi$ such that $F\cap E\neq\emptyset$. Then $E\cap F=\mathbb{P}^{n-3}\subset \mathbb{P}^{n-2}=F$. Take $C_2=\mathbb{P}^1\subset F$ and intersect both side of \cref{eqxi} with $C_2$. We have that $\xi \cdot C_2=1$ and $\pi^{\ast}[\mathcal{O}_{\mathbb{P}^2}(1)]\cdot C_2=0$ and $E\cdot C_2=1$. Thus we get that $m=\frac{1}{2}$. Hence we know that $\mathrm{Cl}(X)\otimes \mathbb{Q}=\mathbb{Q}\cdot\frac{1}{2}[L]$. Let $D$ be a component of $\Delta$. Suppose that
\begin{equation*}
(\psi_{|\xi|})_{\ast}^{-1}[D]=m_1\pi^{\ast}[(\mathcal{O}_{\mathbb{P}^2}(1))]+m_2 [\xi]
\end{equation*}
for some natural numbers $m_1$, $m_2$. We have that $D=(\psi_{|\xi|})_{\ast}(\psi_{|\xi|})_{\ast}^{-1}D\sim_{\mathbb{Q}}(\dfrac{m_1}{2}+m_2)L$. Hence $\frac{m_1}{2}+m_2\geq \frac{1}{2}$. Being a generalized cone, $X$ is $\mathbb{Q}$-factorial. For the  $\mathbb{Q}$-Cartier divisor $K_X$ we have
\begin{equation*}
\label{lceq}
K_X=(\psi_{|\xi|})_{\ast}(K_T)=(\psi_{|\xi|})_{\ast}(\pi^{\ast}(\mathcal{O}(-1))-(n-1)\xi)=-(n-\frac{1}{2})[L].
\end{equation*}
Now $K_X+D+(n-1)L\equiv_{\mathrm{num}}(\dfrac{m_1-1}{2}+m_2)L$ is pseudo-effective. We thus exclude case $(c4)$.

\end{proof}

\begin{lemma}
\label{dimension-control}
Let $(X,D)=(\mathbb{P}(\mathcal{V}),D)$ be a log canonical pair, where $\pi:\mathbb{P}(\mathcal{V})\rightarrow C$ is a projective bundle over a smooth curve $C$ and $\mathcal{V}$ is an ample vector bundle of rank $n$. If for a general fiber $F$, we have that $D|_F=\mathcal{O}_F(d)$ for some $d>0$. Then $\dim(X)=\dim(\mathbb{P}(\mathcal{V}))=2$.
 \end{lemma}
 \begin{proof}
We have that 
\begin{center}
 $(K_{\mathbb{P}(\mathcal{V})}+D+(n-1)\xi)|_F=(\pi^{\ast}(K_C+\mathrm{det}(\mathcal{V}))+D-\xi)|_F=\mathcal{O}_F(d-1)$.
 \end{center}
   We take a \textit{thrifty dlt modification} for $(\mathbb{P}(\mathcal{V}),D)$ as in \citep[Corollary~1.36.]{kol13}, \textit{i.e.}, a proper birational morphism $f:\mathbb{P}(\mathcal{V})^{\mathrm{dlt}}\rightarrow \mathbb{P}(\mathcal{V})$ with a boundary divisor $\Delta^{\mathrm{dlt}}$ such that:
\begin{enumerate}
\item $(\mathbb{P}(\mathcal{V})^{\mathrm{dlt}},\Delta^{\mathrm{dlt}})$ has dlt singularities;
\item $f^{\ast}(K_{\mathbb{P}(\mathcal{V})}+D)\sim_{\mathbb{Q}}K_{\mathbb{P}(\mathcal{V})^{\mathrm{dlt}}}+\Delta^{\mathrm{dlt}}$;
\item $K_{\mathbb{P}(\mathcal{V})^{\mathrm{dlt}}}+\Delta^{\mathrm{dlt}}$ is $f$-nef;
\item $\mathbb{P}(\mathcal{V})^{\mathrm{dlt}}$ is $\mathbb{Q}$-factorial.
\end{enumerate}

 Thus we have that
\begin{center}
$\xymatrix{\mathbb{P}(\mathcal{V})^{\mathrm{dlt}}\ar[r]^f&\mathbb{P}(\mathcal{V})\ar[r]^\pi&C}$.
\end{center}
We set that $g=\pi\circ f$ and $\xi'=f^{\ast}\xi$. Then we have that
\begin{center}
$f^{\ast}(K_{\mathbb{P}(\mathcal{V})}+D+(n-1)\xi)\equiv_{\mathrm{num}}K_{\mathbb{P}(\mathcal{V})^{\mathrm{dlt}}}+\Delta^{\mathrm{dlt}}+(n-1)\xi'$
\end{center} 

As $f$ is surjective, we have that $f_{\ast}$ preserves numerical equivalence. By the projection formula we have that
\begin{center}
$f_{\ast}(K_{\mathbb{P}(\mathcal{V})^{\mathrm{dlt}}}+\Delta^{\mathrm{dlt}}+(n-1)\xi') \equiv_{\mathrm{num}}f_{\ast}f^{\ast}(K_{\mathbb{P}(\mathcal{V})}+D+(n-1)\xi)=K_{\mathbb{P}(\mathcal{V})}+D+(n-1)\xi$.
\end{center}
 Hence $K_{\mathbb{P}(\mathcal{V})^{\mathrm{dlt}}}+\Delta^{\mathrm{dlt}}+(n-1)\xi'$ cannot be pseudo-effective.
So there exists an extremal ray $R$ of $\NE(\mathbb{P}(\mathcal{V})^{\mathrm{dlt}})$ such that
\begin{equation*}
(K_{\mathbb{P}(\mathcal{V})^{\mathrm{dlt}}}+\Delta^{\mathrm{dlt}}+(n-1)\xi')\cdot R <0.
\end{equation*}
For $0<\epsilon\ll 1$, we have that
\begin{equation}
\label{dim-est}
(K_{\mathbb{P}(\mathcal{V})^{\mathrm{dlt}}}+(1-\epsilon)\Delta^{\mathrm{dlt}}+(n-1)\xi')\cdot R <0.
\end{equation}

We note that $\xi'\cdot R=f^{\ast}(L)\cdot R>0$, for otherwise any curve $l$ such that $[l]\in R$ is contracted by $f$, which means $(K_{\mathbb{P}(\mathcal{V})^{\mathrm{dlt}}}+\Delta^{\mathrm{dlt}})\cdot R\geq 0$, a contradiction. Hence $R$ is in fact a $(K_{\mathbb{P}(\mathcal{V})^{\mathrm{dlt}}}+\Delta^{\mathrm{dlt}})$-negative extremal ray. 

By the contraction theorem (\emph{cf.} \citep[Theorem~1.1.(4)]{F11}, we get the contraction morphism $\mathrm{cont}_R:\mathbb{P}(\mathcal{V})^{\mathrm{dlt}}\rightarrow Y$ which contracts the ray $R$. Let $S\subset \mathbb{P}(\mathcal{V})^{\mathrm{dlt}}$ be a fiber of $\mathrm{cont}_R$. If $\mathrm{dim}(S)\geq 2$, there exists a curve $l\subset S$ that is contracted to a point by $g$. As $K_{\mathbb{P}(\mathcal{V})^{\mathrm{dlt}}}+\Delta^{\mathrm{dlt}}$ is $f$-nef, the $K_{\mathbb{P}(\mathcal{V})^{\mathrm{dlt}}}+\Delta^{\mathrm{dlt}}$-negative curve $l$ can not be contracted to a point by $f$. Hence $l$ maps finitely onto a curve $l'\subset F$. Now we have that
\begin{align*}
(K_{\mathbb{P}(\mathcal{V})^{\mathrm{dlt}}}+\Delta^{\mathrm{dlt}}+(n-1)\xi')\cdot l &=(f^{\ast}(K_{\mathbb{P}(\mathcal{V})}+D+(n-1)\xi))\cdot l\\
&=\mathrm{deg}(l/l')(K_{\mathbb{P}(\mathcal{V})}+D+(n-1)\xi)\cdot l\\
&=\mathrm{deg}(l/l')\mathcal{O}_F(d-1)\cdot l'\\
&\geq 0,
\end{align*}
a contradiction. Hence any fiber of $\mathrm{cont}_R$ has dimension at most $1$.

Let $E\subset \mathrm{exc}(\mathrm{cont}_R)$ be an irreducible component of the exceptional locus of $\mathrm{cont}_R$. We thus have that
\begin{center}
$\mathrm{dim}(E)-\mathrm{dim}(\mathrm{cont}_R(E))\leq 1$.
\end{center}

For $0<\epsilon\ll 1$, the pair $(\mathbb{P}(\mathcal{V})^{\mathrm{dlt}},(1-\epsilon)\Delta^{\mathrm{dlt}})$ has klt singularities (\emph{cf.} \citep[Proposition~2.41.]{KM98}). For  small $\epsilon$, the  divisor $-(K_{\mathbb{P}(\mathcal{V})^{\mathrm{dlt}}}+(1-\epsilon) \Delta^{\mathrm{dlt}})$ is still $\mathrm{cont}_R$-ample.
The estimate of the length of extremal ray by Kawamata (\emph{cf.} \citep[Theorem~7.46.]{deb}) for klt pairs shows that the rational curves $l\in R$ cover $E$ and there exists a rational curve $l_{\epsilon}\in R$ such that 
\begin{center}
$0<-(K_{\mathbb{P}(\mathcal{V})^{\mathrm{dlt}}}+(1-\epsilon)\Delta^{\mathrm{dlt}})\cdot l_{\epsilon}\leq 2$.
\end{center}  
 For any curve $l$ whose class $[l]$ is in $R$, we have that $\xi'\cdot l\geq 1$. Combining these two inequalities with \cref{dim-est}, we have that
\begin{equation*}
0>(K_{\mathbb{P}(\mathcal{V})^{\mathrm{dlt}}}+(1-\epsilon)\Delta^{\mathrm{dlt}}+(n-1)\xi')\cdot l_{\epsilon}\geq -2+(n-1).
\end{equation*}
Hence $n=2$. 
\end{proof}

\begin{lemma}
\label{dim-2-class}
Set $(X,L):=(\mathbb{P}(\mathcal{V}),\mathcal{O}_{\mathbb{P}(\mathcal{V})}(1))$, where $\mathcal{V}$ is a rank $2$ ample vector bundle over $\mathbb{P}^1$. Suppose that $\Delta$ is a reduced  divisor on $X$ and $K_X+\Delta+L$ is not pseudo-effective. Then we have one of the following:
\begin{enumerate}
\item Either $\Delta=\sum D_i$ where $D_i\cong \mathbb{P}^1$ are distinct fibers of the structure map $\pi:\mathbb{P}(\mathcal{V})\rightarrow \mathbb{P}^1$; or
\item $(X,L)=(\mathbb{P}(\mathcal{O}_{\mathbb{P}^1}(a)\oplus\mathcal{O}_{\mathbb{P}^1}(1)),\mathcal{O}_{\mathbb{P}(\mathcal{O}_{\mathbb{P}^1}(a)\oplus\mathcal{O}_{\mathbb{P}^1}(1))}(1))$ with $a>1$ and $D$ is the unique section of $\mathbb{P}(\mathcal{O}_{\mathbb{P}^1}(a)\oplus\mathcal{O}_{\mathbb{P}^1}(1))\rightarrow \mathbb{P}^1$ such that 
\begin{center}
$D\equiv_{\mathrm{num}}\mathcal{O}_{\mathbb{P}(\mathcal{O}_{\mathbb{P}^1}(a)\oplus\mathcal{O}_{\mathbb{P}^1}(1))}(1))-af$,
\end{center}
 where $f$ is a general fiber; or
\item $\mathcal{V}=\mathcal{O}_{\mathbb{P}}(1)\oplus \mathcal{O}_{\mathbb{P}^1}(1)$.
\end{enumerate}
\end{lemma}

\begin{proof}
 As $\mathcal{V}$ is ample, we know that $\mathcal{V}\cong\mathcal{O}_{\mathbb{P}^1}(a)\oplus\mathcal{O}_{\mathbb{P}^1}(b)$ with $a,b > 0$. We may suppose that $a\geq b>0$. If $a=b=1$, then $(X,L)\cong (\mathbb{P}^1\times \mathbb{P}^1, (1,1))$ which is a rank $4$ hyperquadric in $\mathbb{P}^3$. Thus we are in case $(3)$ of \cref{dim-2-class}.

Hence in the rest we only consider $a>b$. We follow the convention in \citep[Notation~V.2.8.1]{har78} in this proof. Set $e:=a-b\geq 0$. Set $\mathcal{W}:=\mathcal{V}\otimes \mathcal{O}_{\mathbb{P}^1}(-a)$. We have that $X_{e}:=\mathbb{P}(\mathcal{W})\cong \mathbb{P}(\mathcal{V})$. 
We denote by $p:\mathbb{P}(\mathcal{W})\rightarrow \mathbb{P}^1$ the projection. By \citep[Lemma~II.7.9]{har78}, we know that $(X,L)\cong (X_e, \mathcal{O}_{X_e}(1)\otimes p^{\ast}(\mathcal{O}_{\mathbb{P}^1}(a)))$.
  We denote the general fiber of $p$ by $f'$. Note that $\mathcal{W}$ satisfies the assumption in \citep[Proposition~2.8.]{har78}. Hence there exists a section $C_0$ of $p$ such that $\mathcal{O}_{X_e}(1)\cong \mathcal{O}_{X_e}( C_0)$. \citep[Proposition~2.9.]{har78} implies $C_0^2=-e$. Hence if $a\neq b$, we have that $C_0$ is unique.
 We know that $L\equiv_{\mathrm{num}} C_0+af'$ and \citep[Lemma~2.10.]{har78} implies that $K_{X_e}\sim -2C_0+(-2-e)f'$. 
 Let $D$ be a component of $\Delta$. Then $K_X+D+L$ is not pseudo-effective. Assume that $D\sim xC_0+d'f'$, with $x,d'$ being integers. We have that
 \begin{equation*}
 K_X+L+D \num (x-1)C_0+(d'+b-2)f'.
\end{equation*} 
 As $D$ is a prime divisor, \citep[Corollary~V.2.18-(b)]{har78} implies one of the following:
\begin{enumerate}
\item[(i)] $x=0, d'=1$, and $K_X+L+D \equiv_{\mathrm{num}} -C_0+(b-1)f'\notin \mathrm{Pseff}(X)$;
\item[(ii)] $x=1, d'=0$ and $K_X+L+D \equiv_{\mathrm{num}} (b-2)f'$, which is not pseudo-effective if and only if $b=1$;
\item[(iii)] $x>0, d'>xe$. Note that $d'+b-2\geq 0$. So we have that $K_X+L+D$, being a positive combination of effective divisors,  is pseudo-effective;
\item[(iv)] $e>0, x>0,$ and $d'=xe$. Again we have that $d'+b-2\geq 0$. So $K_X+L+D$, being a positive combination of effective divisors,  is pseudo-effective.
\end{enumerate}

In case $(i)$, the divisor $D$ is a fiber of $p$, which maps isomorphically to a fiber of $\pi$ under the canonical isomorphism $\mathbb{P}(\mathcal{W})\cong \mathbb{P}(\mathcal{V})$.

In case $(ii)$, as $D$ is irreducible, \citep[Proposition~V.2.20-(a)]{har78} implies $D=C_0$. Hence $D$ is the unique section of $\pi:\mathbb{P}(\mathcal{O}_{\mathbb{P}^1}(a)\oplus\mathcal{O}_{\mathbb{P}^1}(1))\rightarrow \mathbb{P}^1$ such that $D\equiv_{\mathrm{num}} \xi-\pi^{\ast}(\mathcal{O}_{\mathbb{P}^1}(a))$. 
If we have another component $D'$ of $\Delta$, we know that $D'\num f$. Then 
\[K_X+L+D+D'\num 0\]
is pseudo-effective, a contradiction. Thus $\Delta=D$ is irreducible and we are in case $(2)$ of \cref{dim-2-class}.

 Write $\Delta=\sum D_i$. If we don't have any component $D$ of $\Delta$ such that $D\sim C_0$, each $D_i$ will be a fiber. Then
\[K_X+L+\Delta\num -C_0+(b-2+k)f'\]
 is not pseudo-effective, where $k$ is the number of components of $\Delta$. Hence we are in case $(1)$ of \cref{dim-2-class}.
\end{proof}

\section{Semi-log canonical polarized varieties}
\label{slcpv}
The hypothesis for the pair $(X,\Delta)$ in \cref{Clr} alludes to the normalization of a slc variety together with its conductor divisors. In this section, we will show how to use \cref{Clr} to classifying polarized slc varieties.

For the basic definition and statements of slc varieties we refer to \citep[Chapter~5]{kol13}. See also \citep[Chapter~2.5]{Liu-thesis} for an account in our setup.

We recall the definition of \textit{conductor}.

\begin{definition}[conductor]
\label{cond}
Let $X$ be a reduced scheme and $\pi:\bar{X}\rightarrow X$ its normalization. The \textit{conductor ideal}
\begin{center}
$\cond_X:=\Hom_{\mathcal{O}_X}(\pi_{\ast}\mathcal{O}_{\bar{X}},\mathcal{O}_X)$
\end{center}
is the largest ideal sheaf of $\mathcal{O}_X$ such that it is also an ideal sheaf of $\pi_{\ast}\mathcal{O}_{\bar{X}}$. As $\pi$ is finite, we have a unique ideal sheaf $\cond_{\bar{X}}$ of $\bar{X}$ that corresponds to $\cond_X$.

We define the conductor schemes to be 
\begin{center}
$D:=\mathrm{Spec}(\mathcal{O}_X/\cond_X)$ and $\bar{D}:=\mathrm{Spec}(\mathcal{O}_{\bar{X}}/\cond_{\bar{X}})$
\end{center}
They fit into the Cartesian square
\begin{center}
$\xymatrix{\bar{D}\ar[d]\ar[r]&\bar{X}\ar[d] \\
D\ar[r]& X}.$
\end{center}
\end{definition}

Note that when $X$ is demi-normal, the conductors $D$ and $\bar{D}$ are reduced divisors. We now give our definition of slc singularities.
\begin{definition}[\protect{\citep[Definition-Lemma~5.10]{kol13}}]
\label{slc}
Let $(X,\Delta)$ be a pair with $X$ demi-normal. Let $\pi:\bar{X}\rightarrow X$ be its normalization, the conductors $\bar{D}$ and $D$ as in \cref{cond}. The pair $(X,\Delta)$ is called \textit{semi-log canonical} or \textit{slc} if $(\bar{X},\bar{D}+\bar{\Delta})$ is log canonical.
\end{definition}

\subsection*{Proof of \cref{Slccase}}
\begin{proof}
We know by definition that $(\bar{X},\bar{D})$ is log canonical. Note that the absolute normalization $\pi:\bar{X}\rightarrow X$ is finite (\emph{cf.}  \cite[\href{https://stacks.math.columbia.edu/tag/0BXR}{Tag 0BXR}]{stacks-project}). Hence $\pi^{\ast}(L)$ is ample.  We have that
\begin{center}
$\pi^{\ast}(K_X+(n-1)L)=K_{\bar{X}}+\bar{D}+(n-1)\pi^{\ast}(L)$.
\end{center}

Let $C\subset X$ be a movable curve in $X$ such that $(K_X+(n-1)L)\cdot C<0$. Let $C'\subset \bar{X}$ be a movable curve that dominates $C$. Then by the projection formula we have that
\begin{equation*}
K_{\bar{X}}+\bar{D}+(n-1)\pi^{\ast}(L))\cdot C'=\deg(C'/C)(K_X+(n-1))\cdot C<0 .
\end{equation*}
Hence by the BDPP theorem, the divisor $K_{\bar{X}}+\bar{D}+(n-1)\pi^{\ast}(L)$ is not pseudo-effective.

Note that $D$ and $\bar{D}$ are reduced. We denote by $\bar{D}^\nu$, $D^\nu$ respectively their normalizations. Then $\pi$ induces a degree $2$ map $\nu:\bar{D}^\nu\rightarrow D^\nu$ and there is a Galois involution $\tau:\bar{D}^\nu\rightarrow \bar{D}^\nu$ which is generically fixed point free (\emph{cf.} \citep[5.2]{kol13}).  Thus we have the following diagram
\begin{center}
$\xymatrix{\bar{D}^\nu \ar@(ul,dl)_{\tau}\ar[d]^\nu \ar[r] &\bar{D} \ar[d] \ar[r] & \bar{X} \ar[d]^{\pi}\\
D^\nu\ar[r] & D\ar[r] & X}$.
\end{center}

Since $\nu:\bar{D}^\nu \rightarrow D^\nu$ has degree $2$, we have by the projection formula that
\begin{center}
$\pi^{\ast}(L)|_{\bar{D}^\nu}^{n-1}=\mathrm{deg}(\nu)\cdot (L|_{D^\nu}^{n-1})\in 2\mathbb{Z}$.
\end{center}

We now apply \cref{Clr} to $(\bar{X},\bar{D},\pi^{\ast}(L))$. We have one of the following:
\begin{enumerate}
\item $(\bar{X},\pi^{\ast}L)\cong (\mathbb{P}^n,\mathcal{O}_{\mathbb{P}^n}(1))$,. The conductor $\bar{D}= H$ is a prime divisor where $H$ is a hyperplane of $ \mathbb{P}^n$;

\item[(2.i)] There is a $(\mathbb{P}^{n-1},\mathcal{O}_{\mathbb{P}^{n-1}}(1))$-bundle $(\mathbb{P}(E),\mathcal{O}_{\mathbb{P}(E)}(1))$ over a smooth curve $C$, and a birational morphism $\mu:\mathbb{P}(E)\rightarrow \bar{X}$ such that $\mu^{\ast}(\pi^{\ast}L)\cong \mathcal{O}_{\mathbb{P}(E)}(1)$ and $\bar{D}=\sum F_i$ is a finite sum where $F_i\cong\mu(\mathbb{P}^{n-1})$ are images of distinct general fibers by $\mu$ and $\deg(\mathbb{P}^{n-1}/F_i)=1$ ;

\item[(2.ii)] $(\bar{X},\pi^{\ast}L)=(\mathbb{P}(\mathcal{O}_{\mathbb{P}^1}(a)\oplus\mathcal{O}_{\mathbb{P}^1}(1)),\mathcal{O}_{\mathbb{P}(\mathcal{O}_{\mathbb{P}^1}(a)\oplus\mathcal{O}_{\mathbb{P}^1}(1))}(1))$ with $a>1$ and $\bar{D}=C$, where  $C$ is the unique section of $\mathbb{P}(\mathcal{O}_{\mathbb{P}^1}(a)\oplus\mathcal{O}_{\mathbb{P}^1}(1))\rightarrow \mathbb{P}^1$ such that 
\begin{center}
$C\equiv_{\mathrm{num}}\mathcal{O}_{\mathbb{P}(\mathcal{O}_{\mathbb{P}^1}(a)\oplus\mathcal{O}_{\mathbb{P}^1}(1))}(1))-af$,
\end{center}
 where $f$ is a general fiber;

\item[(3.i)] $(\bar{X},\pi^{\ast}L)\cong (Q,\mathcal{O}_{\mathbb{P}^{n+1}}(1))$, where $Q\subset \mathbb{P}^{n+1}$ is a $\mathrm{rk}(Q)=3$ hyperquadric, the divisor $\bar{D}$ is a hyperplane in $Q$ and  $[\bar{D}]=\dfrac{1}{2}[H\cap Q]$ where $H$ is a hyperplane in $\mathbb{P}^{n+1}$;

\item[(3.ii)] $(\bar{X},\pi^{\ast}L)\cong (Q,\mathcal{O}_{\mathbb{P}^{n+1}}(1))$, where $Q\subset \mathbb{P}^{n+1}$ is a $\mathrm{rk}(Q)=4$ hyperquadirc. If we write $Q=\mathrm{Proj}\left( \dfrac{\mathbb{C}[x_0, \ldots, x_{n+1}]}{(x_0 x_1-x_2 x_3)}\right) $, then $\bar{D}$ is prime and $\bar{D}$ is the cone with vertex $\mathbb{P}^{n-3}$ over $\mathbb{P}^{1}\times \mathrm{pt}$ or $\mathrm{pt}\times \mathbb{P}^1$. In particular, $\bar{D}\cong \mathbb{P}^{n-1}$.
\end{enumerate}

In case $(1)$, we have that $\bar{D}^\nu=\bar{D}\cong \mathbb{P}^{n-1}$ is smooth and $\pi^{\ast}(L)|_{\bar{D}^\nu}=\mathcal{O}_{\mathbb{P}^{n-1}}(1)$. As $\pi^{\ast}(L)|_{\bar{D}^\nu}^{n-1}=1$ is odd. We exclude case $(1)$.

In case $(2.i)$, we have that $\bar{D}=\sum_{1\leq i\leq k} F_i$ for a natural number $k$ and the morphism 
$\mu:\coprod_{1\leq i\leq k}\mathbb{P}^{n-1}\rightarrow \bar{D}$ factors through $\bar{D}^\nu\rightarrow \bar{D}$ (\emph{cf.} \cite[\href{https://stacks.math.columbia.edu/tag/035Q}{Tag 035Q}]{stacks-project}-$(4)$). 
Hence $\pi^{\ast}(L)|_{\bar{D}^\nu}^{n-1}=k$ and $k$ is even. As $\deg({\mathbb{P}^{n-1}/F_i})=1$, we have that 
\begin{equation*}
\pi^{\ast}(L)|_{F_i}^{n-1}=\deg({\mathbb{P}^{n-1}/F_i})(\mathcal{O}_{\mathbb{P}(E)}(1)|_{\mathbb{P}^{n-1}})^{n-1}=1.
\end{equation*}
Thus each irreducible component of $D$ has pre-image consisting of two of the $F_i$'s. We have thus the diagram
\begin{center}
$\xymatrix{\mathbb{P}(E)\ar[r]^{\mu}\ar[d]_p &\bar{X}\ar[r]^{\pi}&X\\
C}$.
\end{center}

Set $k=2m$. We write $D=\sum_{1\leq i\leq m} D_i$. We denote the two components of $\bar{D}$ that are mapped onto $D_i$ by $F_{i,1}$ and $F_{i_2}$. Let $x_{i,1}$ (\emph{resp.} $x_{i,2}$) be the point of $C$ such that $\mu(p^{-1}(x_{i,1}))=F_{i,1}$ (\emph{resp.} $\mu(p^{-1}(x_{i,2}))=F_{i,2}$). As $C$ is smooth, we may glue $x_{i,1}$ and $x_{i,2}$. We thus get a nodal curve $C'$ together with a quotient morphism $q:C\rightarrow C'$ such that there exists a rank $n$ vector bundle $E'$ on $C'$ satisfying $q^{\ast}(E')=E$. The morphism $\pi\circ \mu$ thus factors through $\mathbb{P}(E)$, \textit{i.e.} we have the following commutative diagram:
\begin{center}
$\xymatrix{
\mathbb{P}(E)\ar[r]^{\mu}\ar[d] &\bar{X}\ar[d]^{\pi}\\
\mathbb{P}(E')\ar[r]^{\mu'}&X
}$.
\end{center}
The morphism $\mu'$ is birational. If we denote $x_i=p(x_{i,1})$ and $F_i$ the fiber of $x_i$ in $\mathbb{P}(E')$, we have that $D_i=\mu'(F_i)$. Thus we have the result of \cref{Slccase}.

In case $(2.ii)$, we have that $\pi^{\ast}(L)\cdot C=a-e=1$. Hence we exclude this case.

In case $(3.i)$, the conductor $\bar{D}$ is irreducible and $\pi^{\ast}(L)|_{\bar{D}^n}^{n-1}=1$. Hence we also exclude this case.

In case $(3.ii)$, the conductor $\bar{D}\cong \mathbb{P}^{n-1}$ and $\pi^{\ast}(L)|_{\bar{D}}=\mathcal{O}_{\mathbb{P}^{n-1}}(1)$ which is not divisible by $2$. Hence we exclude this case, too.
\end{proof}

\bibliographystyle{alpha}
\bibliography{polarised-slc-varieties-with-high-nef-value}

\end{document}